\documentclass[reqno, 11pt]{amsart}

\usepackage{url}
\usepackage[colorlinks, linkcolor=blue, citecolor=blue, urlcolor=blue]{hyperref}

\usepackage{times}

\newtheorem{theorem}{Theorem}[section]
\newtheorem{lemma}[theorem]{Lemma}

\newtheorem{proposition}[theorem]{Proposition}
\newtheorem{corollary}[theorem]{Corollary}

\theoremstyle{definition}
\newtheorem{definition}[theorem]{Definition}
\newtheorem{ex}[theorem]{Example}

\newtheorem{remark}[theorem]{Remark}

\numberwithin{equation}{section}

\usepackage{bbm}
\usepackage{euscript}
\usepackage{pb-diagram}
\usepackage{lamsarrow}
\usepackage{pb-lams}
\usepackage{amsmath}
\usepackage{amsthm}

\usepackage{mathrsfs}
\usepackage{amsxtra}
\usepackage{amssymb}
\usepackage{pifont}
\usepackage{amsbsy}

\usepackage{graphicx}
\usepackage{epstopdf}

\newskip\aline \newskip\halfaline
\def\skipaline{\vskip\aline}

\def\qedbox{$\rlap{$\sqcap$}\sqcup$}
\def\qed{\nobreak\hfill\penalty250 \hbox{}\nobreak\hfill\qedbox\skipaline}

\def\proofend{\eqno{\mbox{\qedbox}}}

\newcommand{\one}{\mathbbm{1}}

\newcommand{\ur}{\underline{\mathbb R}}

\newcommand\bR{{\mathbb R}}

\newcommand\bZ{{\mathbb Z}}

\DeclareMathOperator{\tr}{{\rm tr}}

 \DeclareMathOperator{\Hom}{Hom}
 \DeclareMathOperator{\End}{End}

\DeclareMathOperator{\spa}{span}

\DeclareMathOperator{\ev}{\mathbf{ev}}

\DeclareMathOperator{\pf}{\mathbf{Pf}}
\DeclareMathOperator{\ppf}{\mathbf{pf}}

\DeclareMathOperator{\cov}{\boldsymbol{cov}}

\DeclareMathOperator{\Skew}{Skew}
\DeclareMathOperator{\Jac}{Jac}

\newcommand{\be}{{\boldsymbol{e}}}

\newcommand{\bp}{{\boldsymbol{p}}}

\newcommand{\bu}{{\boldsymbol{u}}}

\newcommand{\bv}{{\boldsymbol{v}}}
\newcommand{\bw}{{\boldsymbol{w}}}
\newcommand{\bx}{{\boldsymbol{x}}}
\newcommand{\by}{{\boldsymbol{y}}}

\newcommand{\bsE}{{\boldsymbol{E}}}

\newcommand{\bsP}{\boldsymbol{P}}

\newcommand{\bsU}{{\boldsymbol{U}}}
\newcommand{\bsV}{{\boldsymbol{V}}}

\newcommand{\bgamma}{{\boldsymbol{\gamma}}}
\newcommand{\bGamma}{{\boldsymbol{\Gamma}}}

\newcommand{\bXi}{\boldsymbol{\Xi}}

\newcommand{\si}{{\sigma}}

\newcommand{\eps}{{\epsilon}}
\newcommand{\vfi}{{\varphi}}

\newcommand{\eA}{\EuScript{A}}
\newcommand{\eB}{\EuScript{B}}

\newcommand{\eH}{\EuScript H}

\newcommand{\eK}{\EuScript{K}}

\newcommand{\eO}{\EuScript{O}}

\newcommand{\eS}{\EuScript{S}}
\newcommand{\eT}{\EuScript{T}}

\newcommand{\ra}{\rightarrow}
\newcommand{\hra}{\hookrightarrow}

\newcommand{\Llra}{{\Longleftrightarrow}}

\newcommand{\lan}{\langle}
\newcommand{\ran}{\rangle}

\def\inpr{\mathbin{\hbox to 6pt{\vrule height0.4pt width5pt depth0pt \kern-.4pt \vrule height6pt width0.4pt depth0pt\hss}}}

\newcommand{\pa}{\partial}

\newcommand{\ori}{\boldsymbol{or}}

\newcommand{\ube}{\underline{\boldsymbol{e}}}
\newcommand{\ubf}{\underline{\boldsymbol{f}}}

\begin{document}

\title[Stochastic Gauss-Bonnet-Chern]{A stochastic Gauss-Bonnet-Chern formula} 

% This is the American J. of Math Version

%\date{Started  March 5, 2014. Completed  on August 20, 2014. Last modified on {\today}. }

\author{Liviu I. Nicolaescu}
%\thanks{This work was partially supported by the NSF grant, DMS-1005745.}

\address{Department of Mathematics, University of Notre Dame, Notre Dame, IN 46556-4618.}
\email{nicolaescu.1@nd.edu}
\urladdr{\url{http://www.nd.edu/~lnicolae/}}


\begin{abstract} We prove that  a Gaussian ensemble of smooth random sections of a real vector bundle $E$ over compact manifold $M$  canonically defines a metric  on $E$ together with a connection compatible  with it. Additionally, we prove a refined Gauss-Bonnet theorem  stating that if  the bundle $E$ and the manifold $M$ are oriented,  then the Euler form of the above connection   can be identified, as a  current, with the expectation of the random current defined by the  zero-locus of a   random section  in the above  Gaussian  ensemble. \end{abstract}

\maketitle

\tableofcontents

\section{Introduction}
\setcounter{equation}{0}

\subsection{Notation and terminology}   Suppose that $X$ is a smooth manifold. For any  vector space $V$,  we denote by $\underline{V}_X$ the trivial bundle $V\times X\to X$.

 We denote by   $|\Lambda_X|\to X$  the  line bundle of $1$-densities on $X$, \cite{GS, N1}, so  that we have a well defined integration  map
\[
\int_X : C_0^\infty(|\Lambda_X|)\to \bR,\;\; C^\infty(|\Lambda_X|)\ni \rho\mapsto \int_X\rho(dx) .
\]
Suppose that  $F$ is a smooth vector  bundle  over $X$. We have  two natural projections 
\[
\pi_x,\pi_y: X\times X\to X,\;\;\pi_x(x,y)=x,\;\;\pi_y(x,y)=y,\;\;\forall x,y\in X.
\]
We set $F\boxtimes F:= \pi_x^*F\otimes \pi_y^* F$, so that $F\boxtimes F$ is vector bundle over $X\times X$.

 Following \cite[Chap.VI,\S1]{GS},   we define  a  \emph{generalized section} of $F$ to be a     continuous linear functional on the space $C_0^\infty(F^*\otimes |\Lambda_X|)$ equipped with the natural locally convex topology.  We denote by $C^{-\infty}(F)$ the space of  generalized   sections of $F$.    We have a natural injection,  \cite[Chap.VI, \S1]{GS}  
 \[
 i:C^\infty(F)\hra C^{-\infty}(F).
 \]
 Recall that a  Borel probability   measure $\mu$ on $\bR$ is called  (centered) \emph{Gaussian} if  has the form
 \[
 \mu(dx)=\bgamma_v(dx):=\begin{cases}
 \frac{1}{\sqrt{2\pi v} }e^{-\frac{x^2}{2v}} dx, & v>0,\\
 \delta_0, & v=0.
 \end{cases}
 \]
 where $\delta_0$ denotes the Dirac measure concentrated at the origin.

\subsection{Gaussian ensembles of sections and correlators}    The concept of Gaussian smooth random section of a vector bundle is very similar to the better known   concept of Gaussian  random function.     Throughout this paper  we fix a smooth compact connected manifold $M$ of dimension $m$ and a smooth real vector bundle $E\to M$ of  rank $r$.    

The notion of random section  of a vector bundle  is  described    in great detail in \cite[Sec. 8]{Bax}. This description  relies on the concept  of abstract Wiener space  due to L. Gross, \cite{Gross}.   Since this concept  may be less familiar to the readers with a more geometric bias,  we  decided to include an alternative  approach, hopefully more palatable  to  geometers.   From this point of view, a  random smooth section of $E$ is identified with a  probability measure on the space of  generalized  sections $C^{-\infty}(E)$ supported on the subspace $C^\infty(E)$.    The construction of such measures relies on the  fundamental work of  R.A. Minlos \cite{Min},  Gelfand-Vilenkin \cite{GeVi2}, X. Fernique \cite{Fer} and  L. Schwartz \cite{Sch}.     We  describe below   the results relevant  to our  main investigation.

The space $C^\infty(E^*\otimes|\Lambda_M|)$ is a nuclear countable Hilbert space in the sense of \cite{GeVi2} and, as such,  its dual  $C^{-\infty}(E)$ satisfies several useful measure theoretic properties.    The next result  follows from \cite{Fer}.
\begin{proposition}
\begin{enumerate}

\item The $\si$-algebra of  weakly Borel subsets of $C^{-\infty}(E)$ is equal to the $\si$-algebra of strongly Borel subsets.  We will refer to this $\si$-algebra  as the \emph{Borel $\si$-algebra} of $C^{-\infty}(E)$.

\item Every Borel probability measure on $C^{-\infty}(E)$ is Radon.

\item Any Borel subset of $C^\infty(E)$  (with its natural topology) belongs to the Borel $\si$-algebra of $C^{-\infty}(E)$.
\end{enumerate}
\label{prop: borel}
\end{proposition}

Any  section $\vfi\in C^\infty(E^*\otimes |\Lambda_M|)$ defines a continuous linear map $L_\vfi: C^{-\infty}(E)\to\bR$. Following \cite{Bog, GeVi2} we define   a  \emph{centered Gaussian measure}\footnote{In the sequel, for simplicity, we will drop the attribute  \emph{centered} when referring  to the various Gaussian measures since we will be working exclusively  with such objects.} on $C^{-\infty}(E)$ to be  a Borel probability measure  $\bGamma$  such that, for any section  $\vfi\in C^\infty(E^*\otimes |\Lambda_M|)$  the pushforward  $(L_\vfi)_\#(\bGamma)$ is a centered Gaussian $\bgamma_\vfi$ measure on $\bR$. 

The  measure $\bGamma$  is completely determined by its \emph{covariance form} which is the symmetric, nonnegative definite  bilinear map
\[
\eK_{\bGamma}:  C^\infty(E^*\otimes |\Lambda_M|)\times C^\infty(E^*\otimes |\Lambda_M|)\to \bR
\]
given by
\[
\eK_\bGamma(\vfi,\psi)=\bsE_\bGamma\bigl(\,L_\vfi\cdot L_\psi\,\bigr),\;\;\forall\vfi,\psi\in C^\infty(E^*\otimes |\Lambda_M|).
\]
Above,  $\bsE_\bGamma$ denotes the expectation with respect to  the  probability measure $\bGamma$ and we interpreted  $L_\vfi, L_\psi$ as  random variables on the probability space $(C^{-\infty}(E),\Gamma)$, 

Results of Fernique \cite[Thm.II.2.3 + Thm.II.3.2]{Fer} imply  that $\eK_\Gamma$ is   separately continuous. According   to  Schwartz' kernel theorem \cite[Chap.I, \S3.5]{GeVi2} the covariance form can be  identified  with a  linear functional $C_\bGamma$ on the topological  vector space
\[
C^\infty\bigl(\, (E^*\otimes |\Lambda_M|)\boxtimes (E^*\otimes |\Lambda_M|)\,\bigr)=C^\infty\bigl(\, (E^*\boxtimes E^*)\otimes |\Lambda_{M\times M}|\,\bigr),
\]
i.e.,  $C_\bGamma\in C^{-\infty}(E\boxtimes E)$. We will refer to $C_\bGamma$ as the \emph{covariance kernel} of $\bGamma$. 

\begin{theorem}[Minlos, \cite{Min}]\label{th: minlos} Given a  generalized section $C\in C^{-\infty}(E\boxtimes E)$ such that   the associated bilinear form 
\[
\eK:  C^\infty(E^*\otimes |\Lambda_M|)\times C^\infty(E^*\otimes |\Lambda_M|)\to \bR
\]
is symmetric and nonnegative definite, there exists a  unique  Gaussian measure on $C^{-\infty}(E)$ with covariance kernel $C$.
\end{theorem}

 \begin{definition}  A     Gaussian   measure $\bGamma$ on $C^{-\infty} (E)$ is called \emph{smooth} if   $C_\bGamma$ is given by a  smooth section of $E\boxtimes E$. We will refer to it as the \emph{covariance density}.    We will refer to   the smooth Gaussian measures  on $C^{-\infty}(E)$ as  a \emph{Gaussian  ensemble of smooth  sections} of $E$.   \qed
 \end{definition}
 
 A smooth section $C$ of $E\boxtimes E$ can be viewed  as a smooth  family of bilinear maps
 \[
 \tilde{C}_{\bx,\by}: E^*_\bx\times E^*_\by\to \bR,\;\;\bx,\by\in M,
 \]
 given by
 \[
 \tilde{C}_{\bx,\by}(\bu^*,\bv^*):=\bigl\lan\, \bu^*\otimes\bv^*, C_{\bx,\by}\,\bigr\ran,\;\;\forall \bu^*\in E_\bx^*,\;\;\bv^*\in E_\by^*,
 \]
 where $\lan-,-\ran$ denotes the natural pairing between a vector space and its dual.  In the sequel we will  identify $C_{\bx,\by}$ with the associated bilinear map $\tilde{C}_{\bx,\by}$.

The next result, proved in Appendix \ref{s: tech},    explains the  role of the smoothness   condition. 

 \begin{proposition} If the    Gaussian measure $\bGamma$ on $C^{-\infty}(E)$ is smooth, then $\bGamma\bigl(\, C^\infty(E)\,\bigr)=1$ .  In other words, a random  generalized section in the Gaussian ensemble determined by $\bGamma$ is a.s. smooth. 
 \label{prop: assm}
 \end{proposition}

 Using  Propositions \ref{prop: borel} and  \ref{prop: assm} we deduce  that a smooth  Gaussian measure  on $C^{-\infty}(E)$ induces a Borel probability measure on $C^\infty(E)$.   Observe also that, for any $\bx\in M$, the induced map
 \[
 C^\infty(E)\to E_\bx,\;\;C^\infty(E)\ni\vfi\mapsto \vfi(\bx)\in E_\bx
 \]
 is Borel measurable.    The next result, proved in Appendix \ref{s: tech},  shows that   the collection of  random variables $(\,\vfi(\bx)\,)_{\bx\in M}$ is Gaussian. 
 
 \begin{proposition}  Suppose that $\bGamma$ is  a smooth Gaussian measure on $E$  with  covariance density $C$. Let $n$ be a positive integer.  Then for any points $\bx_1,\dotsc, \bx_n\in M$ and any  $\bu_i^*\in E_{\bx_i}^*$, $i=1,\dotsc, n$  the random vector
 \[
 C^\infty(E)\ni \vfi\mapsto  \bigl(\,X_1(\vfi),\dotsc, X_n(\vfi)\,\bigr)\in\bR^n,\;\; X_i(\vfi) :=\lan\, \bu_i^*,\vfi(\bx_i)\,\ran,\;\;i=1,\dotsc,n,
 \]
 is Gaussian. Moreover
 \begin{equation}
 \bsE(X_iX_j) =C_{\bx_i,\bx_j}(\bu_i^*,\bu_j^*),\;\;\forall i,j.
 \label{eq: cov_dens}
 \end{equation}
 \label{prop: gauss_vec}
 \end{proposition}

 A section $C\in C^\infty(E\boxtimes E)$ is called  \emph{symmetric} if
 \[
 C_{\bx,\by}(\bu^*,\bv^*)=C_{\by,\bx}(\bv^*,\bu^*),\;\;\forall \bx,\by\in M,\;\;\forall \bu^*\in E_\bx^*,\;\;\bv^*\in E_\by^*.
 \]
If $C$  is the covariance   density of  a  smooth Gaussian measure $\Gamma$ on $C^{-\infty}(E)$, then  Proposition  \ref{prop: gauss_vec} shows that $C$ is symmetric.  

A symmetric section  section  $C\in C^\infty(E\boxtimes E)$ is  called  \emph{nonnegative/positive definite} if all  the symmetric  bilinear forms $C_{\bx,\bx}$ are such. Clearly  the covariance density  of a smooth Gaussian measure $\Gamma$ on $C^{-\infty}(E)$ is symmetric and nonnegative definite.   

 \begin{definition}  (a) A \emph{correlator} on $E$ is a section $C\in C^\infty(E\boxtimes E)$  which is symmetric and  nonnegative definite.  The correlator  is called \emph{nondegenerate}   if it is \emph{positive} definite. 
 
 \noindent (b) A correlator $C\in C^\infty(E\boxtimes E)$ is called \emph{stochastic} if it is the covariance density of a  Gaussian  ensemble smooth  sections of $E$.
 
 \noindent (c) A  Gaussian ensemble of smooth sections of $E$   is called \emph{nondegenerate} if its covariance density is a \emph{nondegenerate} correlator.   
 \qed
 \end{definition}
 
 \begin{remark} Minlos'  Theorem \ref{th: minlos}  shows that not all correlators are stochastic.   The results in \cite{Bax} show that a correlator  $C$ is stochastic if and only if  it is a \emph{reproducing kernel},  i.e., for any natural number $n$, any points $\bx_1,\dotsc, \bx_n\in M$ and any  $\bu_i^*\in E_{\bx_i}^*$, $i=1,\dotsc, n$,  the  symmetric matrix
 \[
 \bigl(\, C_{\bx_i,\bx_j}(\bu_i^*,\bu_j^*),\bigr)_{1\leq i,j\leq n}
 \]
 is nonnegative definite.\qed
 \end{remark}

 \begin{lemma} There exist nondegenerate  Gaussian ensembles of smooth  sections of $E$.
 \label{lemma: gauss}
 \end{lemma}
 
 \begin{proof}     Fix a finite dimensional  subspace $\bsU\in C^\infty(E)$ which is \emph{ample}, i.e., for any $\bx\in M$ the evaluation map  $\ev_\bx:\bsU\to E_\bx$, $\bu\mapsto\bu(\bx)$ is onto.  (The existence of such spaces is a classical fact, proved e.g. in \cite[Lemma 23.8]{BT}.) By duality we obtain injections $\ev_\bx^*:E_\bx^*\to \bsU^*$.
 
 Fix  an Euclidean inner product $(-,-)_\bsU$  on $\bsU$ and denote by $\bgamma$ the Gaussian measure on $\bsU$ canonically determined by this product.  Its covariance pairing $\bsU^*\times \bsU^*\to \bR$  coincides with  $(-,-)_{\bsU^*}$,  the inner product on $\bsU^*$ induced by $(-,-)_\bsU$. More precisely, this means that for any $\xi,\eta\in \bsU^*$ we have
 \begin{equation}
 (\xi,\eta)_{\bsU^*}=\int_\bsU \lan\xi,s\ran\lan\eta,s\ran \bgamma(ds).
 \label{eq: u*}
 \end{equation}
The measure $\bgamma$ defines   a smooth Gaussian measure $\hat{\bgamma}$  on $C^{-\infty}(E)$ such that $\hat{\bgamma}(\bsU)=1$.  Concretely, $\hat{\bgamma}$  is the pushforward of $\bgamma$ via the natural inclusion $\bsU\hra C^\infty(E)$.   This is a smooth measure. Its covariance  density $C$ is computed as follows:   if $\bx,\by\in M$, $\bu^*\in  E^*_\bx$, $\bv^*\in E^*_\by$, then
 \[
 C_{\bx,\by}(\bu^*,\bv^*)= \int_\bsU \lan\bu^*,s(\bx)\ran\lan \bv^*,s(\by)\ran \gamma(ds)= \int_\bsU \lan\ev_\bx^*\bu^*,s\ran\lan\ev^*_\by\bv^*,s\ran \bgamma(ds) 
 \]
 \[
 \stackrel{(\ref{eq: u*})}{=} (\ev^*_\bx\bu^*,\ev^*_\by\bv^*).
 \]
 In particular, when $\bx=\by$ we observe that $C_{\bx,\bx}$ coincides with the restriction to $E_\bx^*$ of the inner product $(-,-)_{\bsU^*}$ so  the form $C_{\bx,\bx}$ is positive definite.
 \end{proof}
 
 \begin{definition} A  Gaussian ensemble of smooth   sections of $E$  with associated Gaussian measure $\bGamma$  on $C^{-\infty}(E)$  is said to have \emph{finite-type} if there exists a finite dimensional  subspace $\bsU\subset C^\infty(E)$ such that $\bGamma(\bsU)=1$. \qed
 \end{definition}
 
 \begin{remark} The Gaussian ensemble constructed in  Lemma \ref{lemma: gauss} has finite type. All the nondegenerate finite type Gaussian  ensembles of  smooth  sections can be obtained in this fashion and, as explained  in \cite{ELL, NS}, the results of Narasimhan and Ramanan \cite{NR1} show that any  pair (metric, compatible connection) on $E$        is determined by the correlator of a   finite type ensemble of smooth sections on $E$.
 
  However, there exist  nondegenerate gaussian ensembles which are not of finite type.  They can be constructed  using an approach  conceptually similar to the one we used in  Lemma \ref{lemma: gauss}. The only difference is that instead of  a finite dimensional  ample space of sections  we use an ample Banach space of $C^k$ sections equipped with a Gaussian measure. For details we refer to \cite{Bax}. \qed
 \end{remark}

 \begin{definition}\label{rem: transversal} A Gaussian ensemble of smooth sections of $E$ is called \emph{transversal}, if  a random section of this  ensemble is a.s. transversal to the zero section of $E$.\qed
 \end{definition}
 
 For the proof of the next  result we refer to Appendix \ref{s: tech}. 
 
 \begin{proposition}\label{prop: transversal}  Any  nondegenerate   ensemble of smooth sections of $E$ is transversal.
 \end{proposition}

\subsection{Statements of the main results}  The main goal of this paper is to investigate some of the rich geometry of a nondegenerate Gaussian ensemble of smooth sections of $E$.      By definition, the correlator  $C$ of such an  ensemble defines a metric  on the dual bundle $E^*$, and thus on $E$ as well.  In  Example \ref{ex: corr}   we illustrate these abstract  constructions  on some familiar situations.

Less obvious  is the fact that,  in general  a nondegenerate correlator  $C$, \emph{not necessarily stochastic},  induces a connection $\nabla^C$  on $E$  compatible with the canonical metric defined by  $C$.  We will refer to this    metric/connection as the correlator metric/connection. We prove this fact  in Proposition \ref{prop: corr_conn}. 

This  connection  depends only on the first order jet of $C$ along the diagonal of $M\times M$.    Using the correlator metric   we can identify the bilinear form $C_{\bx,\by}$ with a linear map $T_{\bx,\by}: E_\by\to E_\bx$.  The definition of the connection shows that its infinitesimal parallel transport    is given by the first order jet of $T_{\bx,\by}$ along the diagonal $\bx=\by$.     

If the correlator    $C$ is  stochastic,  then the connection $\nabla^C$ and    its curvature can be given a probabilistic interpretations.     Proposition \ref{prop: gauss} gives a purely probabilistic  description of its curvature.   This result  contains as a special case Gauss'  Theorema Egregium.

\begin{remark}\label{rem: cor_con} The construction  of $\nabla^C$ in Proposition \ref{prop: corr_conn} feels very classical, but we were   not able to  trace any reference.  In the special case  when $C$ is   a  stochastic   correlator,  this connection is the $L$-$W$ connection in  \cite[Prop. 1.1.1]{ELL}. It can be given  a probabilistic description, \cite[Prop. 1.1.3]{ELL},  or a more geometric  description   obtained by using the stochastic correlator to canonically embed $E$ in a  trivial Hilbert bundle. These constructions  use in an essential way the    reproducing kernel property of a stochastic correlator.  Proposition  \ref{prop: corr_conn} shows that we need a lot less to produce a metric and a compatible connection.  \qed
\end{remark}

Section \ref{s: sgb} contains the main result of this paper, Theorem \ref{th: sgb}.       In this section we need to assume that both $M$ and $E$ are oriented, and the rank of  $E$ is even and  not greater than the dimension of $M$.    Let us digress to recall the classical Gauss-Bonnet-Chern theorem.  

The  Chern-Weil construction associates to  a metric on $E$ and connection $\nabla$ compatible  with this metric  an \emph{Euler form}  $\be(E,\nabla)\in \Omega^r(M)$; see  \cite[Chap.8]{N1}. This form is closed, and  the  cohomology class it determines called  is the \textbf{\emph{geometric}} \emph{Euler class} of $E$.  This cohomology class is  independent of the choice of metric and compatible connection.

If $\bv$  is a  smooth section of $E$ transversal to the zero section,  then its zero locus $Z_{\bv}$ is a   compact, codimension $r$-submanifold of $M$ equipped with a canonical orientation. As such,  it defines  a \emph{closed} integration current $[Z_{\bv}]$ of dimension $(m-r)$ whose  homology class  is independent of   the choice of transversal  section $\bv$.  This means that, if $\bv_0,\bv_1$ are two sections of $E$ transversal to the zero section, then 
\begin{equation}\label{homo}
\int_{Z_{\bv_0}}\eta=\int_{Z_{\bv_1}}\eta,\;\;\forall \eta \in \Omega^{m-r}(M)\;\;\mbox{such that}\;\;\boxed{d\eta=0}.
\end{equation}
Indeed,  using Sard's theorem as  in  \cite[Chap.5, Lemma 2]{Miln},  we can   find an \emph{oriented} $C^1$-submanifold with boundary $\hat{Z}\subset [0,1]\times M$ such that,  $\pa \hat{Z}\cap(\{i\} \times M)=\{i\}\times Z_{\bv_i}$, $i=0,1$, and we have an equality of \emph{oriented} manifolds, $\pa\hat{Z}= Z_{\bv_1}\sqcup-Z_{\bv_0}$. The equality (\ref{homo}) now follows from Stokes' theorem.

The\textbf{ \emph{topological}} \emph{Euler class} of the real, oriented vector bundle $E$ is the cohomology class  of $M$ defined as the pullback  of the Thom class of $E$ via a(ny) section of this bundle. Equivalently, the \emph{topological Euler class} of $E$ is equal to  Poincar\'{e} dual of the homology class of $M$ determined by the zero-locus current $[Z_{\bv_0}]$  defined by  a transversal section $\bv_0$; see   \cite[Exercise 8.3.21]{N1}.

The classical Gauss-Bonnet-Chern theorem  states that the  \emph{topological Euler class} of $E$ is equal to the \emph{geometric Euler class}; see \cite[Chap.IV, Thm.1.51]{HL} or \cite[Thm. 8.3.17]{N1}. This means  that,   for any metric  on $E$, and any connection $\nabla^0$  compatible with the metric,  the  closed current  $[Z_{\bv_0}]$ is \emph{homologous} to the closed current defined by the Euler form $\be(E,\nabla^0)$, i.e., 
\begin{equation}
\int_{Z_{\bv_0}}\eta =\int_M \eta\wedge \be(E,\nabla^0),\;\;\forall \eta\in \Omega^{m-r}(M)\;\;\mbox{such that}\;\;\boxed{d\eta=0}.
\label{gbc1}
\end{equation}
The main goal of this paper is to provide a probabilistic refinement of the above equality.

 Fix a  nondegenerate  Gaussian ensemble of smooth sections of $E$. This determines  a metric  and a  compatible connection $\nabla^{\rm stoch}$ on $E$.     It thus  determines an Euler form $\be(E,\nabla^{\rm stoch})$ on $M$. 

Since our ensemble of sections is nondegenerate it is also transversal according to Proposition \ref{prop: transversal} and thus the current $Z_\bu$ will be well defined for almost all $\bu$ in the  ensemble.  In Theorem \ref{th: sgb} we prove a stochastic  Gauss-Bonnet  formula   stating that the expectation of the random current $[Z_\bu]$ is equal to the  current defined by the Euler form $\be(E,\nabla^{\rm stoch})$, i.e.,
\begin{equation}
\bsE\left(\int_{Z_\bu}\eta\right)=\int_M  \eta\wedge \be(E,\nabla^{\rm stoch}),\;\;\forall \eta\in\Omega^{m-r}(M).
\label{gbc2}
\end{equation}
\begin{remark} (a) Let us point out that the cohomological formula  (\ref{gbc1}) is  a consequence of (\ref{gbc2}).   To see this, fix  an arbitrary metric $h$ on $E$ and a connection  $\nabla^0$ compatible with $h$. Denote by $\be(E,\nabla^0)$ the associate Euler form. Next,  fix a smooth section $\bv_0$ of $E$ that is transversal to $0$.
 
 As indicated in Remark \ref{rem: cor_con}, there exists a finite-type  nondegenerate Gaussian ensemble of  smooth sections of $E$  whose  associated metric is $h$ and associated connection is $\nabla^0$, i.e., $\nabla^{\rm stoch}=\nabla^0$.  Then, a.s., a section $\bu$  in this ensemble is smooth and transversal to  $0$. From   (\ref{homo}) we deduce that for any \emph{closed form} $\eta\in \Omega^{m-r}(M)$ we have the   a.s. equality
 \[
 \int_{Z_{\bv_0}} \eta =\int_{Z_{\bu}}\eta.
 \]
By taking the expectations   of  both sides   and then invoking (\ref{gbc2}) we deduce (\ref{gbc1}).

(b) The stochastic formula (\ref{gbc2}) is stronger than the cohomological one because the   Euler \emph{class} of $E$  could be zero (in cohomology), yet there exist metric connections on $E$   whose associated   Euler \emph{forms} are nonzero.

(c) The stochastic Gauss-Bonnet-Chern formula (\ref{gbc2}) has  a local character. It  suffices to prove it  for forms $\eta$ supported  on coordinate neighborhoods over which $E$ is trivializable.  The general case  follows form these special ones by using  partitions of unity and the obvious  linearity in $\eta$ of both sides of (\ref{gbc2}).  This is in fact the strategy we adopt in our proof. 

(d) In Remark \ref{rem: non}(b) we explain   what happens in the case when the  Gaussian  ensemble of random sections is no longer centered,  say $\bsE(\bu)=\bu_0\in C^\infty(E)$.    Formula (\ref{gbc2})   gets replaced  by (\ref{sgbunc}),   where in the right hand side  we  get  a different  term  that explicitly depends on  the geometry of $E$ and the bias  $\bu_0$.\qed
 \end{remark}
 
  We prove the stochastic formula (\ref{gbc2}) by reducing it to the   Kac-Rice formula  \cite[Thm. 6.4,6.10]{AzWs} using  a bit of differential geometry  and  certain   Gaussian  computations  we  borrowed from  \cite{AT}. For the reader's convenience  we  have included  in Appendix \ref{s: pf} a brief survey of these facts.

\subsection{Related results} In our earlier work \cite{NS} we  proved a special case  of this stochastic  Gauss-Bonnet  formula for nondegenerate Gaussian ensembles of finite type.  The proof   in  \cite{NS} is differential geometric in nature and does not extend  to the general situation discussed in the present paper.  

In \cite{N4}  we used  related probabilistic techniques to prove a  cohomological Gauss-Bonnet-Chern  formula of the type (\ref{gbc1}) in the special case   when $E=TM$,  and the connection $\nabla$ is the Levi-Civita connection of a metric on  $TM$.    Still in the case $E=TM$,  one can use  rather different probabilistic ideas (Malliavin calculus) to prove the cohomological Gauss-Bonnet; the case  when $\nabla$ is  the Levi-Civita   connection of a metric  on $M$ was investigated by E. Hsu \cite{Hsu}, while   the  case of a general  metric connection on $TM$ was recently investigated by H. Zhao \cite{Zhao}.

\subsection*{Acknowledgments.}   I want to thank the anonymous   referee for the very constructive and informative comments, suggestions and questions  that  helped improve the quality of  this paper.

\section{The differential geometry of correlators}
\setcounter{equation}{0}

A correlator  on a real vector bundle $E\to M$ naturally induces   additional  geometric structures on $E$. More precisely, we will show  that it induces a metric on $E$ together with a connection compatible with this metric.     Here are  a few circumstances  that lead   to correlators.
 
 \begin{ex}\label{ex: corr}  (a) Suppose that $M$ is a properly embedded submanifold of the Euclidean space $\bsU$. Then the   inner product $(-,-)_\bsU$  on $\bsU$ induces  a correlator $C\in C^\infty(T^*M\boxtimes T^*M)$ defined by the equalities
 \[
 C_{\bx,\by}(X,Y)= (X,Y)_\bsU,\;\;\forall \bx,\by\in M,\;\;X\in T_\bx M\subset \bsU,\;\;Y\in T_\by M\subset \bsU.
 \]
 (b)    For any real vector space $\bsU$  and any smooth manifold $M$ we denote by $\underline{\bsU}_M:=\bsU\times M\to M$ the trivial vector bundle over $M$ with fiber $\bsU$. 
 
 Suppose that $\bsU$ is a real, finite dimensional Euclidean space with inner product $(-,-)$.  This induces an  inner product $(-,-)_*$ on $\bsU^*$. Suppose   that $E\to M$ is a smooth real vector bundle over $M$ and $P:\underline{\bsU}_M\to E$ is a  fiberwise surjective   bundle morphism.  In other words, $E$ is a quotient bundle of a trivial real metric vector bundle. The dual  $ P^*: E^*\to \underline{\bsU}^*_M$ is an injective bundle morphism. Hence $E^*$ is a subbundle of a trivial   metric real vector bundle. 
 
 For any $\bx\in M$ and any $u^*\in E^*_\bx$  we obtain a vector $P_\bx^* u^*\in \bsU^*_\bx=$ the fiber of $\underline{\bsU}^*_M$ at $\bx\in M$. This allows us to  define a  correlator $C\in C^\infty(E\boxtimes E)$   given by
 \[
 C_{\bx_1,\bx_2}(u_1^*,u_2^*)=\bigl(\, P_{\bx_1}^*u_1^*,\;\;P^*_{\bx_2} u_2^*\,\bigr)_*,\;\;\forall \bx_i\in M,\;\;u_i^*\in E^*_{\bx_i}\;\;i=1,2.\proofend
 \]
 \end{ex}     
 
 Observe that, by definition, a correlator $C\in C^\infty(E\boxtimes E)$   induces a metric on $E^*$ and thus, by duality, a metric on $E$.  We will denote   both these metric by $(-,-)_{E^*,C}$ and respectively  $(-, -)_{E,C}$. When no confusion is possible will will drop the subscript $E$ or $E^*$ from the notation.   To simplify the presentation  we    adhere to the following conventions.

 \begin{enumerate}

 \item      We will use the Latin letters  $i,j,k$ to denote indices  in the range $1,\dotsc, m=\dim M$.
 
 \item We will use   Greek letters $\alpha,\beta,\gamma$ to denote indices in the range $1,\dotsc, r= {\rm rank}\,(E)$.
 
 \end{enumerate}

Using the metric $(-,-)_C$ we can identify $C_{\bx,\by}\in E_\bx\otimes E_\by$ with an element of 
 \[
 T_{\bx,\by}\in E_\bx\otimes E_\by^*\cong \Hom(E_\by, E_\bx).
 \] 
 We will refer to $T_{x,y}$ as the \emph{tunneling map}  associated to the correlator $C$. Note that $T_{\bx,\bx}=\one_{E_\bx}$.  If we denote by $T_{\bx,\by}^*\in \Hom(E_\by,E_\bx)$ the adjoint of $T_{\bx,\by}$ with respect to the metric $(-,-)_{E, C}$, then the symmetry  of $C$  implies  that
 \[
 T_{\by,\bx}=T_{\bx,\by}^*.
 \]
 
 \begin{lemma}  Fix a point $\bp_0\in M$ and   local coordinates $(x^i)_{1\leq i\leq m}$ in a neighborhood $\eO$ of $\bp_0$ in $M$.   Suppose that  $ \underline{\be}(x)=(\be_\alpha(x))_{1\leq \alpha\leq r}$ is a  local $(-,-)_C$-orthononomal frame of $E|_{\eO}$. We regard it as an isomorphism of metric bundles $ \ur^r_\eO\to  E|_\eO $. We obtain   a smooth map
 \[
 T(\underline{\be}): \eO\times \eO\ra  \Hom(\bR^r) ,\;\;(x,y)\mapsto T(\ube)_{x,y}=\underline{\be}(x)^{-1}T_{x,y}\underline{\be}(y).
 \]
 Then for any $i=1,\dotsc, m$ the operator 
 \[
 \pa_{x^i} T(\underline{\be})_{x,y}|_{x=y}:\ur^r_{y}\to\ur^r_{y},
 \]
 is skew-symmetric.
 \end{lemma}
 
 \begin{proof} We identify $\eO\times \eO$ with an open neighborhood of $(0,0)\in \bR\times \bR$ with coordinates $(x^i,y^j)$.  Introduce new coordinates $z^i:=x^i-y^i$, $s^j:=x^j+y^j$,  so that $\pa_{x^i}=\pa_{z^i}+\pa_{s^i}$. We view the map $T(\ube)$ as depending on the variables $z,s$.   Note that
\[
T(\ube)_{0,s}=\one,\;\;T(\ube)_{-z,s}= T(\ube)_{z,s}^*,\;\;\forall z,s.
\]
We deduce that
\[
\pa_{s^i}T(\ube)|_{0,s}=\pa_{s^i}T(\ube)|_{0,s}^*=0,
\]
\[
\pa_{x^i} T(\ube)|_{0,s} =\pa_{z^i} T(\ube)|_{0,s}+\pa_{s^i}T(\be)|_{0,s}=\pa_{z^i} T(\ube)|_{0,s},
\]
\[
\bigl(\,\pa_{x^i} T(\ube)|_{0,s}\,\bigr)^*=\pa_{x^i} T(\ube)^*|_{0,s}= -\pa_{z_i} T(\ube)|_{0,s} +\pa_{s^i} T(\ube)|_{0,s}=-\pa_{x^i} T(\ube)|_{0,s}.
\]
\end{proof}

 Given a coordinate neighborhood with coordinates $(x^i)$ and  a local  isomorphism of  metric vector bundles (local orthonormal frame) $\underline{\be}: \ur^r_\eO\to E|_\eO$ as above, we define the skew-symmetric endomorphisms
 \begin{equation}
 \Gamma_i(\underline{\be}) :\ur^r_\eO\to \ur^r_\eO,\;\;i=1,\dotsc, m=\dim M,\;\;\Gamma_i(\underline{\be})_y= -\pa_{x^i}T_{x,y}|_{x=y}.
 \label{eq: can-con1}
 \end{equation}
 We  obtain a $1$-form with matrix coefficients $\Gamma(\underline{\be}):=\sum_i \Gamma_i(\underline{\be}) dy^i$. The operator
 \begin{equation}
 \nabla^{\underline{\be}}= d+\Gamma(\underline{\be})
 \label{eq: corr_con}
 \end{equation}
 is then a connection on $\ur^r_\eO$ compatible with the metric natural metric on this trivial bundle.  The isomorphism  $\ube$ induces     a  metric connection  $\ube_*\nabla^{\ube}$ on $E|_\eO$.
 
 Suppose that $\ubf:\ur^r_\eO\to E|_\eO$ is another orthonormal frame of  $E_\eO$ related to $\ube$ via a  transition map
 \[
 g: \eO\to  O(r),\;\;\ubf=\ube\cdot g.
 \]
 Then 
 \[
 T(\ubf)_{x,y} =g^{-1}(x)T(\ube)_{x,y} g(y).
 \]
 We denote by $d_x$ the differential with  respect to the $x$ variable. We deduce
 \[
 \Gamma(\ubf)_y=-d_x T(\ubf)_{x,y}|_{x=y}
 \]
 \[
 =-\bigl(\,d_xg^{-1}(x)\,\bigr)_{x=y}\cdot \underbrace{T(\ube)_{y,y}}_{=\one}\cdot g(y)- g^{-1}(y)\bigl(\,d_xT(\ube)_{x,y} \bigr)|_{x=y} g(y)
 \]
 \[
 =g^{-1}(y) dg(y) g^{-1}(y)\cdot g(y) + g^{-1}(y) \Gamma(\ube)_y g(y) =g(y)^{-1} dg(y) + g^{-1}(y) \Gamma(\ube)_y g(y).
 \]
Thus
\[
\Gamma(\ube\cdot g)= g^{-1}dg +g^{-1}\Gamma(\ube) g.
\]
This  shows that for any    local orthonormal frames $\ube$, $\ubf$ of $E|_\eO$ we have
\[
\ube_*\nabla^{\ube}=\ubf_*\nabla^{\ubf}.
\]
We have thus proved  the following result.

\begin{proposition}\label{prop: corr_conn} If $E\to M$ is a smooth real vector bundle, then any correlator $C$ on $M$ induces a canonical metric  $(-,-)_{C}$ on $E$ and a connection $\nabla^C$ compatible with this metric.  More explicitly, if $\eO\subset M$ is an coordinate neighborhood on $M$ and $\ube:\ur^r_\eO\to E|_\eO$ is an orthogonal trivialization , then  $\nabla^C$  is  described by 
\[
\nabla^C =d +\sum_i \Gamma_i(\ube) dx^i,
\]
where the skew-symmetric  $r\times r$-matrix $\Gamma_i(\ube)$ is given by (\ref{eq: can-con1}). \qed
\end{proposition}

\begin{remark}\label{rem: corr}  (a)  In the special case described in Example \ref{ex: corr},  the  connection associated to the corresponding correlator  coincides  with the Levi-Civita connection of the  metric induced by the correlator.    As mentioned earlier, for a stochastic correlator   there is an alternative, probabilistic   description  of the associated connection; see  \cite[\S 1.1]{ELL}  for details. 

 (b) Suppose that we fix local coordinates  $(x^i)$ near a point $\bp_0$  such that $x^i(\bp_0)=0$.  We denote by $P_{x,0}$ the parallel transport of $\nabla^C$ from $0$ to $X$ along the line segment  from $0$ to $x$.   Then
\[
P_{0,0}=\one_{E_0}=T_{0,0},\;\;\pa_{x^i}P_{x,0}|_{x=0}=-\Gamma_i(0)=\pa_{x^i,0} T_{x,0}|_{x=0}.
\]
We see that the tunneling map $T_{x,0}$ is a first order approximation at $0$ of the parallel transport map $P_{x,0}$ of the connection $\nabla^C$.

(c) Note that we have proved a slightly stronger result. Suppose that  $E\to M$ is  a real vector bundle equipped with a metric. An integral kernel  on $E$ is a section   $T\in C^\infty(E^*\boxtimes E)$ and defines a smooth family of linear operators $T_{\bx,\by}\in \Hom(E_\by, E_\bx)$, $\bx,\by\in M$.   We say that  an integral  kernel $T$ is a \emph{symmetric tunneling} if 
\[
T_{x,x}= \one_{E_x}, \;\; T^*_{\bx,\by}=T_{\by,\bx},\;\;\forall \bx,\by\in M.
\]
The proof of Proposition  \ref{prop: corr_conn}  shows that any symmetric tunneling on a  metric vector bundle  naturally determines a  connection compatible with the metric. \qed
\end{remark}

For later use, we want to give a more explicit description of the curvature of the connection $\nabla^C$ in the special case when   the correlator  $C$   \emph{stochastic} and thus it is the covariance density of  a nondegenerate Gaussian ensemble of smooth  sections  of $E$.

\begin{proposition} Suppose that $C$ is a stochastic correlator on $E$ defined by the nondegenerate  Gaussian ensemble smooth random sections  of $E$.      Denote by $\bu$ a random section in this ensemble. Fix a point $\bp_0$, local coordinates $(x^i)$ on $M$ near $\bp_0$ such that $x^i(\bp_0)=0$ $\forall i$, and a local $(-,-)_C$-orthonormal  frame  $\bigl(\,\be_\alpha(x) \,\bigr)_{1\leq \alpha\leq r}$  of $E$  in  a neighborhood of $\bp_0$ which is   is   synchronous  at $\bp_0$,
\[
\nabla^C\be_\alpha|_{\bp_0}=0,\;\;\forall\alpha.
\]
  Denote by $F$ the curvature of $\nabla^C$,
\[
F=\sum_{ij} F_{ij}(x) dx^i\wedge dx^j,\;\; F_{ij}(x)\in \End(E_{\bp_0}).
\]
Then $F_{ij}(0)$  is the endomorphism of $E_{\bp_0}$ which  in the frame $\be_\alpha(\bp_0)$  is  described by the $r\times r$ matrix with entries
\begin{equation}
F_{\alpha\beta|ij}(0):= \bsE\bigl(\,\pa_{x^i}u_\alpha(x)\pa_{x^j}u_\beta(x)\,\bigr)|_{x=0}-\bsE\bigl(\,\pa_{x^j}u_\alpha(x)\pa_{x^i}u_\beta(x)\,\bigr)|_{x=0},\;\;1\leq\alpha,\beta\leq r,
\label{F}
\end{equation}
where $u_\alpha(x)$ is the random function 
\[
u_\alpha(x):=\bigl(\,\bu(x),\be_\alpha(x)\,\bigr)_C.
\]
\label{prop: gauss}
\end{proposition}

\begin{proof} The random section $\bu$ has the local description
\[
\bu=\sum_\alpha u_\alpha(x)\be_\alpha(x).
\]
 Then $T(x,y)$ is a linear map $E_y\to E_x$ given by the $r\times r$ matrix
\[
T(x,y)= \bigl(\, T_{\alpha\beta}(x,y)\,\bigr)_{1\leq\alpha,\beta\leq r},\;\; T_{\alpha\beta}(x, y) = \bsE(\, u_\alpha(x)u_\beta(y)\,\bigr).
\]
The  coefficients of the connection $1$-form $\Gamma=\sum_i\Gamma_i dx^i$ are endomorphisms of $E_x$  given by $r\times r$ matrices
\[
\Gamma_i(x)=\bigl( \, \Gamma_{\alpha\beta|i}(x)\,\bigr)_{1\leq \alpha,\beta\leq r}.
\]
More precisely, we have 
\begin{equation}
\Gamma_{\alpha\beta|i}(x) = -\bsE\bigl(\,\pa_{x^i}u_\alpha(x)u_\beta(x)\,\bigr).
\label{gamai_sto}
\end{equation}
Because  the frame $\bigl(\,\be_\alpha(x)\,\bigr)$ is synchronous at $x=0$ we  deduce that, \emph{at $\bp_0$},  we have $\Gamma_i(0)=0$ and 
\[
F(\bp_0)=\sum_{i<j} F_{ij} (x)dx^i\wedge dx^j\in \End(E_{\bp_0})\otimes \Lambda^2 T^*_{\bp_0} M,\;\;  F_{ij}= \pa_{x^i}\Gamma_j(\bp_0)-\pa_{x^j}\Gamma_i(\bp_0).
\]
The coefficients $F_{ij}(x)$ are $r\times r$ matrices  with entries $F_{\alpha\beta|ij}(x)$, $1\leq\alpha,\beta\leq r$.  Moreover
\[
F_{\alpha\beta|ij}(0)=\pa_{x^j}\Gamma_{\alpha\beta|j}(0)-\pa_{x^j}\Gamma_{\alpha\beta|i}(0)
\]
\[
\stackrel{(\ref{gamai_sto})}{=}\pa_{x^j}\bsE\bigl(\,\pa_{x^i}u_\alpha(x)u_\beta(x)\,\bigr)|_{x=0}-\pa_{x^i} \bsE\bigl(\,\pa_{x^j}u_\alpha(x)u_\beta(x)\,\bigr)|_{x=0}
\]
\[
=\bsE\bigl(\,\pa^2_{x^jx^i}u_\alpha(x)u_\beta(x)\,\bigr)|_{x=0}+\bsE\bigl(\,\pa_{x^i}u_\alpha(x)\pa_{x^j}u_\beta(x)\,\bigr)|_{x=0}
\]
\[
-\bsE\bigl(\,\pa^2_{x^ix^j}u_\alpha(x)u_\beta(x)\,\bigr)|_{x=0}-\bsE\bigl(\,\pa_{x^j}u_\alpha(x)\pa_{x^i}u_\beta(x)\,\bigr)|_{x=0}
\]
\[
=\bsE\bigl(\,\pa_{x^i}u_\alpha(x)\pa_{x^j}u_\beta(x)\,\bigr)|_{x=0}-\bsE\bigl(\,\pa_{x^j}u_\alpha(x)\pa_{x^i}u_\beta(x)\,\bigr)|_{x=0}.
\]
\end{proof} 

 \begin{remark} When $C$ is the stochastic correlator  defined in Example \ref{ex: corr}(a), Proposition \ref{prop: gauss}   specializes to Gauss' Theorema Egregium.\qed
 \end{remark}
 
 \begin{corollary} Suppose that $\bu$ is a nondegenerate, Gaussian smooth random section of $E$ with covariance density $C\in C^\infty(E\boxtimes E)$. Denote by $(-,-)_C$  and respectively $\nabla^C$ the metric and respectively the connection on $E$ defined  by $C$. Then for  any $\bp_0\in M$ the random variables  $\bu(\bp_0)$ and $\nabla^C\bu(\bp_0)$ are independent.
 \label{cor: indep}
\end{corollary}

\begin{proof} We continue to use the same notations as in the proof of Proposition \ref{prop: gauss}.  Observe first that 
\[
\bigl(\,\bu(\bp_0),\nabla\bu(\bp_0)\,\bigr)\in E_{\bp_0}\oplus  E_{\bp_0}\otimes T_{\bp_0}^* M,
\]
is a Gaussian random vector. The section $\bu$ has the local description
\[
\bu(x)=\sum_\beta u_\alpha(x)\be_\beta(x).
\]
Then
\[
\nabla^C_{x^i}\bu(\bp_0)=\sum_\alpha \pa_{x^i} u_\alpha(0)\be_\alpha(0),\;\; 0=\Gamma_{\alpha\beta|i}(x) \stackrel{(\ref{gamai_sto})}{=} -\bsE\bigl(\,\pa_{x^i}u_\alpha(0)u_\beta(0)\,\bigr).
\]
Since $(u_\beta (0),\pa_{x^i} u_\alpha(0)\,)$ is a Gaussian vector, we deduce that the random variables $u_\beta (0),\pa_{x^i} u_\alpha(0)$ are independent.
\end{proof}

\begin{remark}  The local definition of  the connection coefficients  $\Gamma_i$  shows that the above   independence result is a special  case  of a well known fact in the theory of Gaussian   random vectors:  if $X,Y$ are finite dimensional Gaussian vectors such that  the direct sum $X\oplus Y$ is also Gaussian,  then for a certain deterministic   linear operator  $A$ the random vector $X-AY$ is independent of $Y$; see e.g. \cite[Prop. 1.2]{AzWs}. More precisely,  this happens when
\[
A=\cov(X,Y)\cdot \cov(Y)^{-1}.
\]
Corollary \ref{cor: indep} follows from this fact applied in the special case  $X=d\bu(0)$ and $Y=\bu(0)$.  

If we use local coordinates  $(x^i)$ and a local orthonormal frame  $(\be_\alpha)$ in a neighborhood $\eO$, then we can view $\eO$ as an open subset $\bR^m$ and  $\bu$ as a   map   $\bu:\eO\to\bR^r$. As such, it has a differential $d\bu(x)$ at any $x\in \eO$. The   formula (\ref{gamai_sto})  defining the coefficients of the correlator connection $\nabla^C$  and the classical regression formula \cite[Prop.1.2]{AzWs} yield the following  a.s. equality:  for any point $x\in\eO$  we have 
\begin{equation}
\nabla^C\bu(x) =d\bu(x)-\bsE\bigl( d\bu(x)\,\vert\; \bu(x)\,\bigr).
\label{regress}
\end{equation}
Above, the notation $\bsE(\,\mathbf{var}\;|\; \mathbf{cond}\;)$ stands for the   conditional expectation  of the variable $\mathbf{var}$ given the conditions $\mathbf{cond}$.  The  above equality   implies immediately that  the random vectors $d\bu(x)$ and $\bu(x)$ are independent.

We want to mention that (\ref{regress}) is a special case of the probabilistic description in   \cite[Prop.1.13]{ELL}.
\qed
\end{remark}

 \section{Kac-Rice implies Gauss-Bonnet-Chern}
 \label{s: sgb}
 \setcounter{equation}{0}
 
 In this section    we will prove  a refined Gauss-Bonnet-Chern equality     involving   a  nondegenerate Gaussian  ensembles  of smooth sections of $E$.  We will   make the following additional assumption.
 
 \begin{itemize}
 
 \item The manifold $M$ is oriented.
 
 \item     The bundle $E$  is oriented and   its rank  is even, $r=2h$.
 
 \item $r\leq m=\dim M$.
 
 \end{itemize}

 \subsection{The setup}  We  denote by $\Omega_k(M)$ the space of $k$-dimensional currents, i.e., the space of  linear maps $\Omega^k(M)\to \bR$ that are continuous with respect to the natural locally convex topology on the space of smooth $k$-forms on $M$. If $C$ is a $k$-current and $\eta$ is a smooth $k$-form, then we denote by $\lan \eta, C\ran$ the value of $C$ at $\eta$.
 
 Suppose that we are given a metric on $E$ and a   connection $\nabla$ compatible with the metric.  Observe that if $\bu:M\to E$    is a smooth section of $E$ transversal to the zero section,  then its zero set $Z_{\bu}$ is a  smooth codimension $r$ submanifold of $M$ and there  is a canonical adjunction isomorphism 
 \[
  \mathfrak{a}_\bu:  T_{Z_\bu}M\to E|_{Z_\bu},
  \]
  where  $T_{Z_\bu}M= TM|_{Z_\bu}/TZ_\bu$ is the normal bundle of$Z_\bu\hra M$. For more details about this map we refer to \cite[Exercise 8.3.21]{N1} or \cite[Sec.2]{NS}.    From the orientability of $M$ and $E$, and from the  adjunction induced isomorphism
  \[
 TM|_{Z_\bu}\cong  E|_{Z_\bu}\oplus TZ_\bu
 \]
 we deduce that $Z_\bu$  is equipped with a natural orientation      uniquely determined by the  equalities
 \[
  \ori (TM|_{Z_\bu})=  \ori(E|_{Z_\bu})\wedge \ori( TZ_\bu) \stackrel{(r\in2\bZ)}{=} \ori( TZ_\bu)\wedge  \ori(E|_{Z_\bu}).
  \]
  Thus, the zero set $Z_\bu$ with this induced orientation defines  an integration current $[Z_\bu]\in \Omega_{m-r}(M)$
  \[
   \Omega^{m-r}(M)\ni \eta\mapsto\lan \eta,[Z_\bu]\ran :=\int_{Z_\bu}\eta.
  \]
  To a metric $(-,-)$  on $E$   and a connection $\nabla$ compatible with the metric we can associate a closed  form
 \[
 \be(E,\nabla)\in \Omega^r(M).
 \]
Its construction involves  the concept of \emph{Pfaffian} discussed in great detail in Appendix \ref{s: pf} and it goes as follows.  

Denote by   $F$ the curvature of $\nabla$ and set
 \[
 \be(E,\nabla):=\frac{1}{(2\pi)^h} \pf(-F)\in \Omega^r(M),
 \]
  where the Pfaffian $\pf(-F)$ has  the following local description. Fix  a positively  oriented, local orthonormal frame $\be_1(x),\dotsc,\be_r(x)$ of $E$ defined on some open coordinate   neighborhood $\eO$ of $M$. Then $F|_{\eO}$ is described by a skew-symmetric $r\times  r$ matrix $(F_{\alpha\beta})_{1\leq \alpha,\beta\leq r}$, where
  \[
  F_{\alpha\beta}\in \Omega^2(\eO),\;\;\forall\alpha,\beta.
  \]
  If we denote by $\eS_r$ the group of permutations of $\{1,\dotsc, r=2h\}$, then
\begin{equation}
\pf\bigl(-F\bigr)=\frac{1}{2^h h!}\sum_{\si\in\eS_r}\eps(\si) F_{\si_1\si_2}\wedge \cdots \wedge F_{\si_{2h-1}\si_{2h}}\in\Omega^{2h}(\eO),
\label{-pf1}
\end{equation}
where $\eps(\si)$ denotes the signature of the permutation $\si\in\eS_r$. 

\begin{remark} The $r\times r$-matrix $(F_{\alpha\beta})$  depends on the choice of positively oriented local orthonormal frame $(\be_\alpha(x))$. However, the Pfaffian $\pf(-F)$ is a degree $r$-form on $\eO$  \emph{that is independent of the choice of positively oriented   local orthonormal frame}.\qed
\label{rem: gauge}
\end{remark}

As explained in \cite[Chap.8]{N1}, the degree $2h$-form $\be(E,\nabla)$ is closed and it is  called the  \emph{Euler form} of the  connection $E$.  Moreover, its DeRham cohomology class is independent of the choice of the metric connection $\nabla$.  The Euler form    defines an $(m-r)$-dimensional current
\[
\be(E,\nabla)^\dag\in \Omega_{m-r}(M),\;\;\Omega^{m-r}(M)\ni \eta \mapsto \lan\eta,\be(E,\nabla)^\dag\ran:=\int_M\eta\wedge \be(E,\nabla).
\]

\subsection{A stochastic Gauss-Bonnet-Chern theorem}  We can now state the main theorem of this paper.

\begin{theorem}[Stochastic Gauss-Bonnet-Chern]  Assume that the manifold $M$ is oriented, the bundle $E$ is oriented and has even rank $r=2h\leq m=\dim M$. Fix a   nondegenerate  Gaussian  ensemble  of smooth sections of $E$.  Denote by $\bu$ a random section of this ensemble, by $C$ the correlator  of this Gaussian ensemble,  by $(-,-)_C$ the metric on $E$ induced by $C$ and by $\nabla$ the connection  on $E$ determined by this correlator.   Then the expectation of the random $(m-r)$-dimensional current $[Z_\bu]$ is equal to the current $\be(E,\nabla)^\dag$, i.e., 
\begin{equation}
\bsE\bigl(\, \lan \eta, [Z_\bu]\ran\,\bigr)=\int_M\eta\wedge \be(E,\nabla),\;\;\forall \eta\in \Omega^{m-r}(M).
\label{sgb}
\end{equation}
\label{th: sgb}
\end{theorem}

\begin{proof}   The linearity in $\eta$ of (\ref{sgb}) shows that it suffices to prove this equality in the special case when  $\eta$ is  compactly supported on a coordinate neighborhood $\eO$   of a point $\bp_0\in M$. Fix coordinates   $x^1,\dotsc, x^m$ on $\eO$  with the following properties.

\begin{itemize}

\item $x^i(\bp_0)=0$, $\forall i=1,\dotsc, m$.

\item  The orientation of $M$ along $\eO$ is given by the top degree form $\omega_\eO:=dx^1\wedge \cdots \wedge dx^m$.

\end{itemize}

Invoking again the linearity in $\eta$ of (\ref{sgb})  we  deduce that it suffices to prove it in the special case  when $\eta$ has the form
\[
\eta=\eta_0 dx^{r+1}\wedge \cdots \wedge dx^m,\;\;\eta_0\in C^\infty_0(\eO).
\]
In other words, we have to prove the equality
\begin{equation}
\bsE\bigl(\,\bigr\lan\, \eta_0dx^{r+1}\wedge \cdots \wedge dx^m, [Z_\bu]\,\bigr\ran\,\bigr)=\int_\eO  \eta_0 dx^{r+1}\wedge \cdots \wedge dx^m\wedge \be(E,\nabla),\;\;\forall \eta_0\in C_0^\infty(\eO).
\label{sgb1}
\end{equation}
For any subset 
\[
I=\{i_1<\cdots <i_k\}\subset \{1,\dotsc, m\}
\]
we write $dx^I:= dx^{i_1}\wedge \cdots \wedge dx^{i_k}$. We set
\[
I_0:=\{1,\dotsc, r\},\;\;J_0:=\{r+1,\dotsc, m\}.
\]
We can rewrite (\ref{sgb1}) in the more compact form
\begin{equation}
\bsE\bigl(\,\bigr\lan\, \eta_0dx^{J_0}, [Z_\bu]\,\bigr\ran\,\bigr)=\int_\eO  \eta_0 dx^{J_0}\wedge \be(E,\nabla),\;\;\forall \eta_0\in C_0^\infty(\eO).
\label{sgb2}
\end{equation}

Fix a local, \emph{positively oriented}, $-(,-)_C$-orthonormal frame $(\be_\alpha(x))_{1\leq \alpha\leq r}$ of $E|_\eO$. The restriction to $\eO$ of the  curvature $F$ of $\nabla$ is then a skew-symmetric $r\times r$-matrix 
\[
F=(F_{\alpha\beta})_{1\leq \alpha,\beta\leq r},\;\; F_{\alpha\beta}\in\Omega^2(\eO),\;\;\forall \alpha,\beta.
\]
Each of the $2$-forms $F_{\alpha\beta}$ admits a unique decomposition
\[
F_{\alpha\beta}=\sum_{1\leq i<j\leq m} F_{\alpha\beta|ij} dx^i\wedge dx^j.
\]
For each subset   $I\subset \{1,\dotsc m\}$ we write
\[
F^I_{\alpha\beta}:= \sum_{\substack{ i<j\\ i,j\in I}} F_{\alpha\beta|ij} dx^i\wedge dx^j\in\Omega^2(\eO).
\]
We denote by $F^I$ the skew-symmetric $r\times r$ matrix with entries $F^I_{\alpha\beta}$.

The degree $r$ form $\pf(-F)$ admits a canonical decomposition
\[
\pf(-F)=\sum_{|I|=r}\pf\bigl(\,-F^I\,\bigr)=\sum_{|I|=r}\ppf(-F)_I \, dx^I,\;\; \ppf(-F)_I\in C^\infty(\eO).
\]
The equality (\ref{sgb2})  is then equivalent to the equality
\begin{equation}
\bsE\bigl(\,\bigr\lan\, \eta_0dx^{J_0}, [Z_\bu]\,\bigr\ran\,\bigr)=\frac{1}{(2\pi)^h}\int_\eO  \eta_0 \ppf(-F)_{I_0}\omega_\eO ,\;\;\forall \eta_0\in C_0^\infty(\eO),
\label{sgb3}
\end{equation}
where we recall that $\omega_\eO= dx^1\wedge \cdots \wedge dx^m$. To prove the above equality we will use  the following two-step strategy.

\medskip

\noindent {\bf Step 1.} Invoke the Kac-Rice formula to express  the left-hand side   of (\ref{sgb3}) as an integral over $\eO$
\[
\bsE\bigl(\,\bigr\lan\, \eta_0dx^{J_0}, [Z_\bu]\,\bigr\ran\,\bigr)=\int_\eO\eta_0(x) \rho(x) \omega_\eO,
\]
where $\rho(x)$ is a certain  smooth function on $\eO$.

\smallskip

\noindent {\bf Step 2.} Use the Gaussian computations  in  Appendix \ref{s: pf} to show that
\[
\rho(x)=\ppf(-F)_{I_0}(x),\;\;\forall x\in \eO.
\]

Let us now implement this strategy.   We view $\eO$ as an open neighborhood of the origin in $\bR^m$ equipped with   the canonical  Euclidean metric and the orientation  given by $\omega_\eO$ . Denote by $E_0$ the fiber of $E$ over the origin. Using the  oriented, orthonormal local frame $(\be_\alpha)$ we can view the restriction  to $\eO$ of the random section   $\bu$ as    smooth Gaussian random map 
\[
\bu: \eO\to E_0\cong \bR^r, \;\; x\mapsto (u_\alpha(x))_{1\leq \alpha\leq r},
\]
where again $\bR^r$ is equipped with   the canonical  Euclidean metric  and  orientation  given by  the volume form
\[
\omega_E= du_1\wedge \cdots \wedge du_r.
\]
The fact that the frame $(\,\be_\alpha(x)\,)$ is orthonormal with respect to the metric $(-,-)_C$     implies that for any $x\in\eO$  the probability distribution  of the random vector $\bu(x)$   is the standard  Gaussian measure on the Euclidean space $\bR^r$.  We denote by $p_{\bu(x)}$ the probability density of this vector so that
\begin{equation}
p_{\bu(x)}(y) =\frac{1}{(2\pi)^h} e^{-\frac{1}{2}|y|^2},\;\;y\in E_0\cong \bR^r,
\label{bu_prob}
\end{equation}
where $|-|$ denotes the  canonical Euclidean norm on $\bR^r$, $h=r/2$. 

The zero set $Z_\bu$ is a.s. a  submanifold of  $\eO$ and, as such, it is equipped with  an induced  Riemann metric with  associated volume density $|dV_{Z_{\bu}}|$.  

Recall that if $T: U\to V$ is a linear map between two Euclidean spaces  such that $\dim U\geq \dim V$, then  its \emph{Jacobian} is the scalar
\[
\Jac_T:=\sqrt{\det (TT^*)}.
\]
We define  the  \emph{Jacobian  at $x\in \eO$} of a smooth map $F:\eO\to E_0$ to be the scalar
\[
J_F(x)=\Jac_{dF(x)}=\sqrt{\det d F(x) dF(x)^*},
\]
where $dF(x):\bR^m\to E_0$ is   the differential $dF(x)$ of $F$ at $x$. We set 
\[
\eT:=\Hom(\bR^m,E_0),
\]
 so that   have a  Gaussian random  map
\[
d\bu: \eO\to \eT,  \;\;x\mapsto d\bu(x).
\]
This random map is a.s. smooth. The  random map 
\[
\eO\to E_0\times \eT,\;\;x\mapsto\bigl ( \bu(x), d\bu(x)\,\bigr),
\]
is also a Gaussian random map.   We  have the following Kac-Rice formula, \cite[Thm. 6.4,  6.10]{AzWs}.

\begin{theorem}[Kac-Rice] Let $g: \eT\to\bR$ be a bounded continuous function. Then, for any  $\lambda_0\in C_0(\eO)$, the random variable
\[
\bu\mapsto \int_{Z_\bu} \lambda_0(x) g\bigl(\,d\bu(x)\,\bigr) |dV_{Z_\bu}(x)|\
\]
is integrable  and 
\begin{subequations}
\begin{equation}
\bsE\left(\int_{Z_\bu} \lambda_0(x) g\bigl(\,d\bu(x)\,\bigr) |dV_{Z_\bu}(x)|\,\right)=\int_\eO\lambda_0(x)\bw(x) \omega_\eO(x),\;\;\forall \lambda_0\in C_0(\eO),
\label{kra}
\end{equation}
\begin{equation}
\begin{split}
\bw(x)=\bsE\Bigl(\, J_\bu(x) g(\,d\bu(x)\,)\,\bigr|\; \bu(x)=0\,\Bigr)p_{\bu(x)}(0)\\ \stackrel{(\ref{bu_prob})}{=}\frac{1}{(2\pi)^h}\bsE\Bigl(\, J_\bu(x) g(\,d\bu(x)\,)\,\bigr|\; \bu(x)=0\,\Bigr).
\end{split}
\label{krb}
\end{equation}
\end{subequations}
 In particular, the function $x\mapsto \lambda_0(x) \rho(x)$ is also integrable.
\label{th: kr}
\end{theorem}

The above  equality  extends  to more general  functions $g$.

\begin{definition} We say that a  bounded  measurable function $g:\eT\to\bR$ is \emph{admissible}  if  there exists a  sequence of  bounded continuous functions $g_n:\eT\to\bR$ with the following properties.

\begin{enumerate}
\item The sequence $g_n$ converges a.e. to $g$.
\item  $\sup_n\Vert g_n\Vert_{L^\infty}<\infty$.\qed
\end{enumerate}
\label{def: adm}
\end{definition}

\begin{lemma}   Theorem \ref{th: kr} continues to hold if $g$ is an admissible  function $\eT\to\bR$.
\end{lemma}

\begin{proof} Fix an admissible  function $g:\eT\to\bR$ and a  sequence  of  bounded measurable functions $g_n:\eT\to\bR$ satisfying the  conditions in Definition \ref{def: adm}.  Set $K:= \sup_n\Vert g_n\Vert_{L^\infty}$. Then
\[
\left|\int_{Z_\bu} \lambda_0(x) g_n\bigl(\,d\bu(x)\,\bigr) |dV_{Z_\bu}(x)|\,\right|\leq  K\int_{Z_\bu} |\lambda_0(x)|  |dV_{Z_\bu}(x)|.
\]
The random variable 
\[
\bu\mapsto K \int_{Z_\bu} |\lambda_0(x)|  |dV_{Z_\bu}(x)| 
\]
is integrable according to the Theorem \ref{th: kr} in the special case  $g\equiv K$ and $\lambda_0=|\lambda_0|$.  The dominated converge theorem implies that
\[
\lim_{n\to\infty} \bsE\left(\int_{Z_\bu} \lambda_0(x) g_n\bigl(\,d\bu(x)\,\bigr) |dV_{Z_\bu}(x)|\,\right)=\bsE\left(\int_{Z_\bu} \lambda_0(x) g\bigl(\,d\bu(x)\,\bigr) |dV_{Z_\bu}(x)|\,\right).
\]
A similar argument shows that
\[
\lim{n\to\infty}\bsE\Bigl(\, J_\bu(x) g(\,d\bu(x)\,)\,\bigr|\; \bu(x)=0\,\Bigr)=\bsE\Bigl(\, J_\bu(x) g(\,d\bu(x)\,)\,\bigr|\; \bu(x)=0\,\Bigr).
\]
\end{proof}

To apply  the  above Kac-Rice formula  we need to express the integral over $Z_\bu$ of a form as an integral of a function with respect to the volume density. More precisely, we seek an equality of the type
\[
\int_{Z_\bu} dx^{J_0} =\int_{Z_\bu}  \eta_0(x) g\bigl(\, d\bu(x)\,\bigr) |dV_{Z_\bu}(x)|,
\]
for some admissible  function $g$. This is achieved in the following technical result whose proof can be found in Appendix \ref{s: tech}.

\begin{lemma}\label{lemma: coa} Suppose that $0$ is a regular value of $\bu$. Set  $u_\alpha(x)= (\bu,\be_\alpha(x)\,)$. Then
\[
dx^{J_0}|_{Z_\bu}=\frac{\Delta_{I_0}(d\bu)}{J_\bu}d V_{Z_u},
\]
 where $J_\bu:\eO\to\bR_{\geq 0}$ is the  Jacobian of $\bu$ and $\Delta_{I_0}(d\bu)$ is the determinant of  the  $r\times r$ matrix $\frac{\pa \bu}{\pa x^{I_0}}$ with entries
  \[
  \frac{\pa u_\alpha}{\pa x^j},\;\;\alpha,j\in I_0.
  \]
\end{lemma}

Any linear map $T\in \eT=\Hom(\bR^m,E_0)$ is represented by an $r\times m$ matrix. For any   subset $J$ of $\{ 1,\dotsc, m\}$ we denote by  $\Delta_J(T)$ the determinant of  the $r\times r$ minor  $T_J$ determined by the columns indexed by $J$. 

Denote by $\eT_*$ the subset of $\eT$ consisting of  surjective linear maps $\bR^m\to E_0$. The complement $\eT\setminus \eT_*$ is a negligible subset of $\eT$.   Observe that $T\in\eT^*\Llra \Jac_T\neq 0$. Define
\[
G: \eT\to \bR,\;\; g(T)=\begin{cases}
\frac{\Delta_{I_0}(T)}{\Jac_T},& T\in \eT_*,\\
0, & T\in \eT\setminus \eT_*.
\end{cases} 
\]
 Lemma \ref{lemma: coa}  shows that that if $0$ is a regular value of $\bu$, then
\[
\int_{Z_\bu} \eta_0 dx^{J_0}= \int_{Z_\bu}  \eta_0(x) G(d\bu(x)) |dV_{Z_\bu}(x)|,\;\;\forall \eta_0\in C_0(\eO).
\]

\begin{lemma}  The measurable function  $G: \eT\to \bR$ is admissible.
\end{lemma}

\begin{proof} We first prove that $G$ is bounded on $\eT_*$. This follows from the classical identity
\[
\Jac_T^2=\sum_{|J|=r}\Delta_{J}(T)^2.
\]
This proves that
\[
\left| \frac{\Delta_{I_0}(T)}{\Jac_T}\right|\leq 1.
\]
Now define 
\[
G_n( T):=\frac{\Delta_{I_0}(T)}{ \sqrt{n^{-2}+\Jac_T^2}},\;\;\forall T\in \eT.
\]
Observe that $G_n(T)\nearrow G(T)$  for  $T\in \eT_*$ as $n\to\infty$ and $\sup_n\Vert G_n\Vert_{L^\infty}\leq 1$.
\end{proof}

We  deduce that
\[
\bsE\Bigl(\, \bigl\lan\, \eta_0(x) dx^{J^0}, [Z_\bu]\,\bigl)=\bsE\left(\int_{Z_\bu} \eta_0(x) G\bigl(\,d\bu(x)\,\bigr) |dV_{Z_\bu}(x)|\,\right)
\]
\[
\stackrel{(\ref{kra})}{=}\frac{1}{(2\pi)^h}\int_\eO \eta_0(x)\bsE\Bigl(\,J_\bu(\bx) G(\,d\bu(x)\,)\; \Bigl|\;\bu(x)=0\,\Bigr) \omega_\eO
\]
\[
=\frac{1}{(2\pi)^h}\int_\eO \eta_0(x)\,\underbrace{\bsE\Bigl(\,\Delta_{I_0}\bigl(\,d\bu(x)\,\bigr)\;\Bigl|\; \bu(x)=0\,\Bigr)}_{=:\rho(x)} \omega_\eO.
\]
We have thus proved  the equality
\begin{equation}
\bsE\Bigl(\, \bigl\lan\, \eta_0(x) dx^{J^0}\, , \, [Z_\bu]\,\bigr\ran\,\Bigl)=\frac{1}{(2\pi)^h}\int_\eO\rho(x) \eta_0(x) \omega_\eO,\;\;\forall \eta_0\in C_0^\infty(\eO).
\label{krgb}
\end{equation}
The density $\rho(x)$ in the right-hand-side of the above equality could a priori depend on the choice of the $(-,-)_C$-orthonormal frame   because it involved the frame dependent matrix $\frac{\pa \bu}{\pa x^{I_0}}$. On the other hand, the left-hand-side of the equality (\ref{krgb}) is plainly frame independent. This shows that the density   $\rho$ is also frame  independent.  To prove  (\ref{sgb3}) and thus Theorem \ref{th: sgb} it suffices to show that
\begin{equation}
\bsE\Bigl(\,\Delta_{I_0}\bigl(\,d\bu(x)\,\bigr)\;\Bigl|\; \bu(x)=0\,\Bigr)= \ppf(-F)_{I_0}(x),\;\;\forall x\in\eO.
\label{sgb4}
\end{equation}
We will prove the above equality for $x=0$.   Both sides are frame invariant  and  thus we are free to choose the frame $(\,\be_\alpha(x)\,)$ as we please.  We assume that it is synchronous at $x=0$, i.e.,
\[
\nabla\be_\alpha(0)=0,\;\;\forall \alpha.
\]
Then $\nabla^C\bu(0)=d\bu(0)$. Corollary \ref{cor: indep}  now implies that the  Gaussian vectors $d\bu(0)$ and $\bu(0)$ are independent. Hence 
\[
\bsE\Bigl(\,\Delta_{I_0}\bigl(\,d\bu(0)\,\bigr)\;\Bigl|\; \bu(0)=0\,\Bigr)=\bsE\Bigl(\,\Delta_{I_0}\bigl(\,d\bu(0)\,\bigr)\,\Bigr),
\]
and thus  we have to prove that
\begin{equation}
\bsE\Bigl(\,\Delta_{I_0}\bigl(\,d\bu(0)\,\bigr)\,\Bigr)= \ppf(-F)_{I_0}(0).
\label{sgb5}
\end{equation}
The random variable $\Delta_{I_0}\bigl(\,d\bu(0)\,\bigr)$ is the determinant of the $r\times r$ Gaussian matrix $S:=\frac{\pa\bu}{\pa x^{I_0}}$  with entries
\begin{equation}\label{S_random}
S_{\alpha i}:=\pa_{x^i} u_\alpha (0),\;\;1\leq \alpha,i\leq r.
\end{equation}
Its statistics are determined by the covariances
\begin{equation}\label{cov_nabla}
K_{\alpha i|\beta j}:= \bsE\bigl(\, S_{\alpha i} S_{\beta j} \,\bigr)=\bsE\bigl(\, \pa_{x^i} u_\alpha (0)\pa_{x^j} u_\beta (0)\,\bigr).
\end{equation}
As in Appendix \ref{s: pf}, we  consider  the  $(2,2)$-double-form
  \[
  \bXi_K=\sum_{\alpha<\beta,\;i<j} \bXi_{\alpha\beta|ij} \bv^\alpha\wedge \bv^\beta\otimes\bv^i\wedge \bv^j\in \Lambda^{2,2}\bsV^*,
  \]
  where
  \[
  \bXi_{\alpha\beta|ij}:=\bigl( K_{\alpha i|\beta j}- K_{\alpha j|\beta i}\,\bigr),\;\;\forall 1\leq \alpha,\beta\leq r,\;\;1\leq i,j\in I_0.
  \]
Then
\begin{equation}
\bsE\Bigl(\,\Delta_{I_0}\bigl(\,d\bu(0)\,\bigr)\,\Bigr)\stackrel{(\ref{deta})}= \frac{1}{h!}  \tr \bXi_K^{\wedge h}.
\label{sgb6}
\end{equation}
Now observe that (\ref{F}) implies that
\[
\bXi_{\alpha\beta|ij}=F_{\alpha\beta|ij}(0)=\forall 1\leq \alpha,\beta\leq r,\;\;1\leq i,j\in I_0.
\]
We deduce that
\[
\bXi_K= \Omega_{-F^{I_0}(0)} \stackrel{(\ref{berez})}{:=} \sum_{\substack{\alpha<\beta,\;
i<j,\\ i,j\in I_0}} F_{\alpha\beta|ij} du_\alpha\wedge du_\beta\otimes dx^i\wedge dx^j.
\]
Using (\ref{pf3}) and  (\ref{pf_ber1}) we deduce
\begin{equation}\label{last}
\begin{split}
\pf(-F^{I_0})_{x=0}= \ppf(-F)_{I_0}(0) \,dx^{I_0}= \frac{1}{h!}\Bigl(\, \tr \Omega_{-F^{I_0}(0)}^{\wedge h}\,\Bigr) dx^{I_0}\\
=\frac{1}{h!} \Bigl(\, \tr \bXi_K^{\wedge h}\,\Bigr) dx^{I_0}\stackrel{(\ref{sgb6})}{=}\bsE\Bigl(\,\Delta_{I_0}\bigl(\,d\bu(0)\,\bigr)\,\Bigr)dx^{I_0}.
\end{split}
\end{equation}
This proves (\ref{sgb5}) and thus completes the proof of Theorem \ref{th: sgb}.
\end{proof}

\begin{remark}\label{rem: non} (a) When  the  rank of $E$ is odd, the \emph{topological} Euler  class with real coefficients is trivial, \cite[Thm. 8.3.17]{N1}.    In this case, if $\bu$ is a section of $E$ transversal to the zero  section, then we have the  equality of currents $[Z_{-\bu}]=-[Z_{\bu}]$. If  $\bu$ is a random section of a  smooth, nondegenerate Gaussian ensemble, then the above equality implies $\bsE([Z_{\bu}])=0$.

 (b)  Theorem \ref{th: sgb} deals with  \emph{centered} Gaussian ensembles of smooth sections of $E$.  However its proof can be easily modified to produce results for \emph{noncentered} ensembles as well. 

Suppose that  $\bu$ is a centered nondegenerate random Gaussian smooth section of $E$ with associated induced metric $h$ and   connection $\nabla$ as in Theorem \ref{th: sgb}.  Fix a smooth section $\bu_0$ of  $E$ and form the noncentered Gaussian   random section  $\bv=\bu_0+\bu$.  Then  $\bv$ is  a.s. transversal to the  zero section and we obtain a  random current $[Z_\bv]$.  

Fix $\bp\in M$  and   define  the spaces of \emph{mixed} double-forms
\[
\Lambda^{j,k}(T^*_\bp M, E_\bp):=\Lambda^j  T^*_\bp M\otimes  \Lambda^k E_\bp.
\]
As in Appendix  \ref{s:  pf}, we  have a natural  associative multiplication
\[
\wedge: \Lambda^{j,k}(T^*_\bp M, E_\bp)\otimes \Lambda^{j',k'}(T^*_\bp M, E_\bp)\to \Lambda^{j+j',j+k'}(T^*_\bp M, E_\bp).
\] 
We   have the mixed  double-forms
\[
\nabla \bu_0(\bp)\in   C^\infty\bigl(\,\Lambda^{1,1}(T^* M, E)\,\Bigr),\;\;F(\nabla)\in C^\infty\bigl(\,\Lambda^{2,2} (T^*M, E)\,\bigl).
\]
Fix a point $\bp\in M$.   Observe that the metric  and orientation on $E_p$ canonically  determine a unit vector  $\omega_{E_\bp}$  of  the top exterior power  $\Lambda^r E^*_\bp$, $r=2h$. The canonical map $\beta:\Lambda^r E_\bp\otimes \Lambda^r E^*_\bp\to\bR$ is an isomorphism, and we obtain   natural  contractions
\[
\Lambda^r E_\bp \to\bR,\;\;\vfi\mapsto\vfi\inpr \omega_{E_\bp}:=\beta\bigl(\,\vfi\otimes \omega_{E_\bp}\,\bigr) \in \bR,
\]
\[
\Lambda^{j,r}(T^*_\bp M, E_\bp)\ni A \mapsto  A\inpr \omega_{E_\bp}\in\Lambda^j T^*_\bp M.
\]
Now define 
\[
\pf(-F(\nabla),\bu_0)_\bp:=\sum_{j=0}^h \frac{(-1)^j}{(2h-2j)! j!} \Bigl( \bigl(\nabla \bu_0(\bp)\bigr)^{\wedge (2h-2j)} \wedge  F(\nabla)^{\wedge j}  \Bigr)\inpr \omega_{E_\bp}\in \Lambda^{2h} T^*_\bp M.
\]
Note that $\pf(-F(\nabla))=\pf(-F(\nabla), \bu_0)_{\bu_0=0}$. The proof of Theorem \ref{th: sgb}  shows  that  
\begin{equation}\label{sgbunc}
\bsE\bigl(\, \lan \eta, [Z_\bv]\ran\,\bigr)= \frac{1}{(2\pi)^h}\int_M \pf(-F(\nabla),\bu_0)\wedge \eta,\;\;\forall \eta\in \Omega^{m-r}(M).
\end{equation}
Indeed, the only modification in the proof  appears  when we   consider the random matrix (\ref{S_random}),
\[
S=\Bigl(\,\frac{\pa v^\alpha}{\pa x^i}(0)\,\Bigr)_{\substack{1\leq\alpha\leq r\\ i\in I_0}},\;\;I_0=\{1,\dotsc, r\}.
\]
In this case $S$  is  no longer a  \emph{centered} Gaussian random matrix.   Its expectation is 
\[
\bsE( S)=\nabla_{I_0} \bu_0(0):=\sum_{i\in I_0} dx^i\otimes\nabla_{x^i}\bu_0(0),
\]
   while its covariances are still  given by  (\ref{cov_nabla}).   The  equality  (\ref{sgbunc})  now follows  by using the same argument as in the last part of the proof of Theorem \ref{th: sgb},  with one notable difference: in the equality (\ref{sgb6})   we must invoke  the (\ref{detamu})  with  $\mu=\nabla_{I_0}\bu_0(0)$.

 \qed
\end{remark}
  
\appendix

\section{Proofs of various technical results}
\label{s: tech}
\setcounter{equation}{0}

  \noindent{\bf Proof of Proposition \ref{prop: assm}.}   Fix a metric $g$ on $M$, a metric and  a compatible connection on $E$.  For each nonnegative integer  $k$ we can define   the Sobolev spaces $\eH_k$  consisting of $L^2$-sections of $E$ whose generalized   derivatives up to order $k$ are $L^2$-sections.  We have a decreasing  sequence of Hilbert spaces   $\eH_0\supset \eH_1\supset  \cdots$  whose intersection is $C^\infty(E)$.  For $k\geq 0$ we denote by $\eH_{-k}$ the topological dual of  $\eH_k$ so that we have  a decreasing family of Hilbert spaces $\cdots \subset \eH_1\subset \eH_0\subset \eH_{-1}\subset \cdots.$
 
%The union  of this family of spaces is $C^{-\infty}(E)$,  and the strong topology on $C^{-\infty}(E)$ is the locally convex inductive limit of this  family.  
 
 The results in \cite{Fer} show  that each  of the  subsets $\eH_k\subset C^{-\infty}(E)$, $k\in\bZ$, is a Borel subset. Using  Minlos's theorem \cite[Sec.4, Thm.2]{Min} we  deduce that if the covariance  kernel $C_\bGamma$ is smooth, then $\bGamma(\eH_k)=1$, $\forall  k\in\bZ$. \qed

\noindent {\bf Proof of Proposition \ref{prop: gauss_vec}.} Fix a Riemann metric $g$ on $M$. For each $i=1,\dotsc, n$  choose a  sequence $(\delta_{\nu,i})_{\nu\geq 0}$ of  smooth functions on $M$  supported in a coordinate neighborhood of $\bx_i$ such that
\[
\lim_{\nu\to \infty} \delta_{\nu,i}|dV_g|=\delta_{\bx_i}=\mbox{the Dirac measure concentrated at  $x_i$} .
\]
Fix trivializations of $E$ near each $\bx_i$. Let $t_1,\dotsc, t_n$.  Now define 
\[
\Phi_\nu=:\sum_{i=1}^n t_i\bu^*_i\otimes \delta_{\nu,i}|dV_g|\in C^\infty(E^*\otimes |\Lambda_M|),
\]
and  form the random variable $C^{-\infty} (E)\ni\vfi\mapsto Y_\nu=Y_\nu(\vfi)=L_{\Phi_\nu}(\vfi)$. This is a Gaussian random variable    with variance
 \[
 \bsE_\Gamma(Y_\nu^2) =\eK_\Gamma(\Phi_\nu,\Phi_\nu)=\sum_{i,j} t_it_j\int_{M\times M}C_{x,y}(\bu^*_i,\bu^*_j)\delta_{\nu,i}(x)\delta_{\nu,j}(y) |dV_g(x)dV_g(y)|.
 \]
 Now observe that
 \[
 \lim_{\nu\to\infty} Y_n(\vfi)=\sum_{i=1}^n t_i X_i(\vfi).
 \]
 We deduce that $Y_\nu$ converges in law to $\sum_{i=1}^nt_iX_i$. In particular, this random variable is Gaussian and its variance is
 \[
\lim_{\nu\to\infty}\bsE(Y_\nu^2)= \lim_{\nu\to\infty}\sum_{i,j} t_it_j\int_{M\times M}C_{x,y}(\bu^*_i,\bu^*_j)\delta_{\nu,i}(x)\delta_{\nu,j}(y) |dV_g(x)dV_g(y)|
\]
\[
=\sum_{i,j}t_it_j C_{\bx_i,\bx_j}(\bu^*_i,\bu^*_j).
  \]
 This completes the proof of  Proposition \ref{prop: gauss_vec}.\qed

\noindent {\bf Proof of Proposition \ref{prop: transversal}.}  Let us observe that   when $m=\dim M={\rm rank}\,E$, then  \cite[Prop. 6.5]{AzWs} shows  that any  nondegenerate Gaussian ensemble of smooth sections of $E$ is transversal.    We will reduce the general case to this special situation. 

The result is  certainly local   so it suffices to consider the case of  nondegenerate Gaussian random maps
\[
 F:\Omega \times B\to\bR^r,\;\;(\omega,\bx)\to  F(\omega, \bx)\in\bR^r,
 \]
 where $\Omega=(\Omega, \eA, \bsP)$ is a probability space, $B$ is the unit open ball in $\bR^m$ and $r<m$.  Throughout we assume that $F$ is a.s. $C^2$.  
 
 Fix  a \emph{finite dimensional} space $\bsV$ of smooth functions $B\to \bR^{m-r}$  satisfying the ampleness condition
  \[
  \forall \bx\in B, \;\; \spa\bigl\{\bv(\bx);\;\;\bv\in\bsV\bigr\}=\bR^{m-r}.
  \]
  Equip $\bsV$ with a nondegenerate Gaussian measure $\bGamma$.  Form  a new probability  space
 \[
 (\widehat{\Omega}, \hat{\eA},\widehat{\bsP}):=(\bsV, \eB, \bGamma)\otimes (\Omega, \eA, \bsP),
 \]
 where $\eB$ denotes the $\si$-algebra of Borel subsets of $\bsV$.   Denote by $\Pi_\Omega$ the natural projection $\widehat{\Omega}\to\Omega$. We consider  a new Gaussian random map
 \[
 \hat{F}:\widehat{\Omega}\times B \to\bR^{m-r}\oplus \bR^r=\bR^m,\;\; \hat{F}(\hat{\omega}, \bx):=  \bv(\bx)\oplus F(\omega,\bx),\;\;\forall \hat{\omega}=(\bv,\omega)\in\bsV\times \Omega.
 \]
 Clearly the random map $\widehat{F}$ is a.s. $C^2$ and nondegenerate.  We  denote by $\widehat{\Omega}_*$ the set of $\hat{\omega}\in \widehat{\Omega}$ such that the map $B\ni \bx \mapsto \hat{F}(\hat{\omega},\bx)\in\bR^m$ is $C^2$ and, for any $\bx\in B$, its differential
 \[
 D_\bx \hat{F}(\hat{\omega},-): T_\bx B\to\bR^m
 \]
is  bijective. The random  field $\hat{F}$ satisfies the assumptions in \cite[Prop. 6.5]{AzWs}  and thus  $\widehat{\bsP}(\widehat{\Omega}_*)=1$. 

If we denote by $\pi_r$ the natural projection $\bR^{m-r}\oplus \bR^r\to \bR^r$ we observe that
\[
\pi_rD_\bx \hat{F}(\hat{\omega},-)= D_\bx F(\omega, -): T_\bx B\to \bR^r.
\]
Hence,  if $\omega\in \Pi_\Omega(\widehat{\Omega})$,  the differential  $D_\bx F(\omega, -): T_\bx B\to \bR^r$ is onto for any $\bx\in B$. Clearly $\bsP\bigl(\, \Pi_\Omega(\widehat{\Omega})\,\bigr) =1$. This proves that the random map $F$ is transversal.\qed

 \noindent {\bf Proof of Lemma \ref{lemma: coa}.}   We follow a strategy similar to the one used  in the proof of \cite[Cor. 2.11]{Ncoarea}. Fix a point $p_0\in Z_{\bu}$. Now choose local coordinates $(t^1,\dotsc, t^m)$ on $\eO$ near $p_0$ and local coordinates $y^1,\dotsc, y^r$ on $E_0$ near $0\in E_0$      with the following properties.
  
  \begin{itemize}
  
  \item In the $(t,y)$-coordinates  the map $\bu$ is given  by the linear projection
  \[
  y^j=t^j, \;\; j=1,\dotsc, r.
  \]
  \item   The orientation of $E_0$ is given by $dy=dy^1\wedge \cdots\wedge dy^r$.
  \end{itemize}
  We set
  \[
   dt^{J_0}:= dt^{r+1}\wedge \cdots \wedge dt^r,\;\;  dt^{I_0}:=dt^1\wedge \cdots \wedge dt^r.
  \]
  The coordinates $t^{J_0}$ can be used as local coordinates  on $Z_\bu$ near $p_0$ and we assume that $dt^{J_0}$  defines the induced orientation of $Z_{\bu}$.  We can then write
  \begin{equation}
  \omega_\eO =\rho_\eO dt^{J_0}\wedge dt^{I_0},\;\; \omega_E = \rho_E  dy=\rho_E dy^1\wedge \cdots\wedge dy^r,\;\; dV_{Z_\bu}= \rho_\bu dt^{J_0},
  \label{coa0}
  \end{equation}
  where $\rho_\eO$, $\rho_E$ and $\rho_\bu$ are positive smooth functions on their respective domains. In the $t$-coordinates  we have 
\[
 dx^{J_0}=\lambda dt^{J_0}+\mbox{other exterior monomials}, 
 \]
  where $\lambda$ is the determinant  of  the $(m-r)\times (m-r)$ matrix $ \frac{\pa x^{J_0}}{\pa t^{J_0}}$  with entries $\frac{\pa x^i}{\pa t^j}$, $i,j\in J_0$. Thus
  \[
  dx^{J_0}|_{Z_\bu}=\lambda dt^{J_0}\stackrel{(\ref{coa0})}{=}\frac{\lambda}{\rho_\bu} dV_{Z_{\bu}}.
  \]
 We have 
  \begin{equation}
  dx^{J_0}\wedge \bu^* \omega_E=\vfi\omega_\eO,\;\; \vfi=\vfi=\Delta_{I_0}(d\bu)=\det\left(\frac{\pa\bu}{\pa x^{I_0}}\right).
  \label{coa1}
  \end{equation}
   On the other hand, 
  \begin{equation}
  dx^{J_0}\wedge \bu^*\omega_E= \rho_Edx^{J_0} \wedge dt^{I_0}\stackrel{(\ref{coa0})}{=} \lambda\rho_E dt^{J_0}\wedge dt^{I_0}=\frac{\lambda\rho_E}{\rho_\eO} \omega_\eO.
  \label{coa2}
  \end{equation}
  Using this in (\ref{coa1}) we deduce
  \begin{equation}
  \vfi=\frac{\lambda\rho_E}{\rho_\eO}.
  \label{coa3}
  \end{equation}
  Now observe that, along $Z_\bu$, we have 
  \[
  \frac{\vfi}{J_\bu} dV_{Z_\bu} \stackrel{(\ref{coa0})}{=}\frac{\vfi}{J_\bu}\rho_\bu dt^{J_0}\stackrel{(\ref{coa3})}{=}\frac{\lambda\rho_E\rho_\bu}{J_\bu\rho_\eO} dt^{J_0}=\frac{\rho_E\rho_\bu}{J_\bu\rho_\eO}dx^{J_0}|_{Z_\bu}.
  \]
  On the other hand,  \cite[Lemma 1.2]{Ncoarea} shows that $\frac{\rho_E\rho_\bu}{J_\bu\rho_\eO}=1$  which proves that
  \[
  dx^{J_0}|_{Z_\bu}=\frac{\vfi}{J_\bu}dV_{Z_\bu}\stackrel{(\ref{coa1})}{=}\frac{\Delta_{I_0}(d\bu)}{J_\bu}.
  \]
\qed
 
 \section{Pfaffians and Gaussian computations}
 \label{s:  pf}
 \setcounter{equation}{0}
 
We collect here a few  facts about Pfaffians   needed in the main body of the paper. 

 Fix a positive even integer $r=2h>0$. Given a commutative $\bR$-algebra $\eA$ we denote by $\Skew_r(\eA)$ the space of skew-symmetric $r\times r$-matrices with entries in $\eA$. The Pfaffian of a matrix $F\in \Skew_r(\eA)$ is a certain universal homogeneous  polynomial   of degree $h=r/2$ in the entries of $F$. More precisely, if we denote by $\eS_r$ the group of permutations of $\{1,\dotsc, r=2h\}$, then
\begin{equation}
\pf\bigl(F\bigr)=\frac{(-1)^h}{2^h h!}\sum_{\si\in\eS_r}\eps(\si) F_{\si_1\si_2} \cdots  F_{\si_{2h-1}\si_{2h}}\in\eA, 
\label{pf1}
\end{equation}
where $\eps(\si)$ denotes the signature of the permutation $\si\in\eS_r$.   The Pfaffian can be given an equivalent alternative description.

Fix an \emph{oriented}  real, $r$-dimensional  Euclidean space $E$ and  an \emph{oriented} orthonormal basis $e_1,\dotsc, e_r$ of $E$. Denote by $e^1,\dotsc, e^r$ the dual  basis  of $E^*$ and consider the  $\eA$-valued   $2$-form    
\begin{equation}
\Omega^E_F=-\sum_{1\leq \alpha <\beta} F_{\alpha\beta} \otimes e^\alpha\wedge e^\beta\in \eA\otimes \Lambda^2E^*,
\label{berez}
\end{equation}
  then the Pfaffian of $F$ is uniquely determined by the equality, \cite[Sec. 2.2.4]{N1},
\begin{equation}
\pf(F)e^1\wedge\cdots \wedge e^r =\frac{1}{h!}(\Omega_F^E)^{\wedge h}\in \eA\otimes \Lambda^{2h} E^*.
\label{pf_ber}
\end{equation}
We are interested only in a certain special case when 
\[
\eA=\Lambda^{\mathrm{even}}\bsV^*=\bigoplus_{2k\leq m} \Lambda^{2k} \bsV^*,
\]
where $\bsV$ is a real Euclidean space of dimension $m\geq r$ and 
\[
F_{\alpha\beta}\in \Lambda^2 \bsV^*,  \;\;\forall 1\leq \alpha,\beta\leq r.
\]
  In this case $\pf(F)\in\Lambda^r \bsV^*$  and has the following alternative description.

Fix  an \emph{orthonormal} basis $\{\bv_1,\dotsc,\bv_m)$ of $\bsV$.  For $1\leq \alpha_1,\alpha_2\leq r$ and $1\leq j_1,j_2\leq m$ we set
\begin{equation}
F_{\alpha_1\alpha_2|j_1j_2}:= F^E_{i_1i_2}(\bv_{j_1},\bv_{j_2}).
\label{not}
\end{equation}
Denote  by $\eS_r'$ the subset of $\eS_r$ consisting  of permutations  $(\si_1,\dotsc,\si_{2h})$ such that
\[
\si_1<\si_2,\;\si_3<\si_4,\;\dotsc ,\;\si_{2h-1}<\si_{2h}.
\]
Then
\begin{equation}
\pf\bigl(F\bigr)\bigl(\,\bv_1,\cdots,\bv_r\,\bigr)= \frac{(-1)^h}{h!}\sum_{\vfi,\si\in\eS_r'} \eps(\si\vfi)F_{\si_1\si_2|\vfi_1\vfi_2}\cdots F_{\si_{2h-1}\si_{2h}|\vfi_{2h-1}\vfi_{2h}}.
\label{pf2}
\end{equation}
For every subset $I=\{i_1<\cdots <i_r\} \subset \{1,\dotsc, m\}$ we write
\[
\bv^{\wedge I}                 =\bv^{i_1}\wedge \cdots \wedge \bv^{i_r},
\]
where $\{\bv^1,\dotsc,\bv^m\}$ is the orthonormal  basis of $\bsV^*$ dual to $\{\bv_1,\dotsc, \bv_m\}$. 
\[
\pf(F)=\sum_{|I|= r} \ppf(F)_I  \bv^{\wedge I}.
\]
For  an ordered multiindex $I$   we denote by $\bsV_I$ the   subspace spanned by $\bv_i$, $i\in I$, and by $F^I_{\alpha\beta}$ the restriction of $F_{\alpha\beta}$ to $\bsV_I$, i.e., 
\[
F^I_{\alpha\beta}=\sum_{\substack{i<j\\i,j\in I}}F^I_{\alpha\beta|ij} \bv^i\wedge \bv^j\in \Lambda^2 \bsV_I^*.
\]
We  denote by $F^I$ the $r\times r$  skew-symmetric matrix with entries $(F^I_{\alpha\beta})_{1\leq\alpha,\beta\leq r}$. Note that for any   subset $I\subset \{1,\dotsc, m\}$ of cardinality $r$ we have 
\begin{equation}
\ppf(F)_I\bv^I= \pf(F^I).
\label{pf3}
\end{equation}
This shows that the computation of the Pfaffians  reduces to the case when $\dim \bsV= r$.    This is what we will assume in the remainder of this section.   We fix an orthonormal basis  $\bv_1,\dotsc,\bv_r$ of $\bsV$ and we denote  by $\bv^1,\dotsc,\bv^r$ the dual basis of $\bsV^*$.

 To proceed further we need to introduce some more terminology.  A  \emph{double-form} on the above Euclidean space $\bsV$ is, by definition, an element of the vector space
\[
\Lambda^{p,q} \bsV^*:=\Lambda^p \bsV^*\otimes \Lambda^q \bsV^*,\;\;p,q\in\bZ_{\geq 0}.
\]
We have an associative product $\wedge: \Lambda^{p,q}\bsV^*\times \Lambda^{p',q'}\bsV^*\to \Lambda^{p+p',q+q'}\bsV^*$ given by
\[
(\omega\otimes \eta)\wedge (\omega'\otimes \eta'):=  (\omega\wedge \omega')\otimes (\eta\wedge \eta'),
\]
for any $\omega\in \Lambda^p\bsV^*$, $\eta\in \Lambda^q\bsV^*$, $\omega'\in \Lambda^{p'}\bsV^*$, $\eta'\in \Lambda^{q'}\bsV^*$. Observe that the metric on $\bsV$     produces an isomorphism
\[
\Lambda^{j,j}\bsV^*\cong \End\bigl(\,\Lambda^j \bsV^*,\Lambda^j \bsV^*\,\bigr),
\]
 and thus we have a well defined trace
 \[
 \tr:\Lambda^{j,j}\bsV^*\to\bR,\;\;\forall j=0,1,\dotsc, r.
 \]
 Observe that  an endomorphism $T$ of $\bsV$ can be identified with the $(1,1)$-double-form
 \[
 \omega _T=\sum_{1\leq \alpha,i\leq r} T_{\alpha i}\bv^\alpha\otimes \bv^i,\;\;T_{\alpha i}=(\bv_\alpha, T\bv_i)_{\bsV}.
 \]
 We then have the equality
 \begin{equation}
 \det T=\frac{1}{r!}\tr \omega_T^{\wedge r}.
 \label{det}
 \end{equation}
 Let us specialize  (\ref{berez}) to the case when $E=\bsV$ and $ e^\alpha=\bv^\alpha$. In particular, this implies that $\bsV$ is oriented by the volume form 
 \[
 \Omega_\bsV:=\bv^1\wedge\cdots\wedge\bv^r.
 \]
   If we write
 \begin{equation}
 \Omega_F=-\sum_{\alpha<\beta} F_{\alpha\beta}\otimes \bv^\alpha\wedge \bv^\beta,
 \label{berez1}
 \end{equation}
 then we observe that $\Omega_F\in \Lambda^{2,2}\bsV^{*,*}$,  and that the equality  (\ref{pf_ber}) can be rewritten in  the more compact form
 \begin{equation}
 \pf(F)=\frac{1}{h!}\bigl(\,\tr \Omega_F^{\wedge h}\,\bigr)\Omega_\bsV.
 \label{pf_ber1}
 \end{equation}
 
 As explained in \cite[\S12.3]{AT},  the formalism  of  double-forms and Pfaffians makes  its appearance in certain Gaussian computation. Suppose that $S$ is a random  Gaussian endomorphism of $\bsV$ with  entries
 \[
 S_{\alpha i}:=(\bv_\alpha, S\bv_i)_\bsV,\;\;\alpha,i=1,\dotsc,r,
 \]
 centered Gaussian random variables with covariances
 \[
 K_{\alpha i|\beta j}:= \bsE\bigl(\, S_{\alpha i} S_{\beta j}\,\bigr),\;\;\forall \alpha,\beta, i,j=1,\dotsc, r.
 \]
 We regard  $S$ as $(1,1)$-double-form 
 \[
 S=\sum_{\alpha, i}S_{\alpha i}\bv^\alpha\otimes\bv^i,
 \]
  and we get a random    $(r,r)$-double-form $S^{\wedge r}\in \Lambda^{r,r} \bsV^*$. Its expectation  can be given a very compact description.     Define the  $(2,2)$-double-form
  \[
  \bXi_K:=\sum_{\alpha<\beta,\;i<j} \bXi_{\alpha\beta|ij} \bv^\alpha\wedge \bv^\beta\otimes\bv^i\wedge \bv^j\in \Lambda^{2,2}\bsV^*,
  \]
  where $\bXi_{\alpha\beta|ij}:=\bigl( K_{\alpha i|\beta j}- K_{\alpha j|\beta i}\,\bigr)$, $\forall \alpha,\beta,i,j$.  Using \cite[Lemma 12.3.1]{AT}, the case $\mu=0$, we deduce
  \begin{equation}
 \frac{1}{r!}\bsE\bigl( \,S^{\wedge r}\,\bigr) =\frac{1}{h!}  \bXi_K^{\wedge h},\;\;\bsE\bigl( \,\det S\,\bigr)= \frac{1}{h!}  \tr \bXi_K^{\wedge h}.
 \label{deta}
 \end{equation}
 More generally, if $\mu\in \Lambda^{1,1} V^*$ is a  fixed (deterministic) $(1,1)$-double form,  then \cite[Lemma 12.3.1]{AT} shows that
 \begin{equation}\label{detamu}
 \bsE\bigl(\,\det(\mu +S)\,)=\frac{1}{r!}\tr\bsE\bigl(\,(\mu+S)^{\wedge 2h}\,\bigr)=\sum_{j=0}^h \frac{1}{(2h-2j)! j!} \tr\Bigl(\mu^{\wedge (2h-2j)}\wedge \bXi_K^{\wedge j}\Bigr).
 \end{equation}

\end{document}